\documentclass[12pt]{amsart}

\usepackage{geometry}
\usepackage[all]{xy}
\usepackage{graphicx}
\usepackage{amssymb}
\usepackage{amsfonts}
\usepackage{epstopdf}
\usepackage{amsmath, amsthm}
\usepackage{mathrsfs}
\usepackage{rotating}

\newcommand{\bb}[1]{\mathbb{#1}}
\newcommand{\cc}[1]{\mathcal{#1}}
\newcommand{\lie}[1]{\mathfrak{#1}}

\newcommand{\diag}{\textrm{diag}}
\newcommand{\eo}[1]{\stackrel{-\circ}{#1}}
\def\inv{^{-1}}

\theoremstyle{plain}
\newtheorem{theorem}{Theorem}[section]

\newtheorem{corollary}[theorem]{Corollary}
\newtheorem{lemma}[theorem]{Lemma}
\newtheorem{proposition}[theorem]{Proposition}

\theoremstyle{definition}
\newtheorem{example}[theorem]{Example}
\newtheorem{remark}[theorem]{Remark}

\newtheorem{definition}[theorem]{Definition}

\numberwithin{equation}{section}

\begin{document}
\title[Lower bounds for Gromov width in the SO(n) coadjoint orbits.]
{Lower bounds for Gromov width in the special orthogonal coadjoint orbits.}

\author{Milena Pabiniak}

\address{Milena Pabiniak,
Department of Mathematics,  
Cornell University, 
Ithaca NY}
\email{milena@math.cornell.edu}

\thanks{\today}

\maketitle

\begin{abstract}Let $G$ be a compact connected Lie group $G$ and $T$ its maximal torus. The coadjoint orbit $\cc{O}_{\lambda}$ through $\lambda \in \lie{t}^*$ is canonically a symplectic manifold. Therefore we can ask the question of its Gromov width.
In many known cases the width is exactly the minimum 
over the set $\{ \langle \alpha_j^{\vee},\lambda \rangle; \alpha_j^{\vee} \textrm{ a coroot, }\langle \alpha_j^{\vee},\lambda \rangle>0\}$. 
We will show that the Gromov width for regular coadjoint orbits of the special orthogonal group is at least this minimum. The proof uses the torus action coming from the Gelfand-Tsetlin system.
\end{abstract}
\tableofcontents
\section{Introduction}In 1985 Mikhail Gromov proved the nonsqueezing theorem which is one of the foundational results in the modern theory of symplectic invariants.
The theorem says that a ball
$B^{2N}(r)$ of radius $r$, in a symplectic vector space $\bb{R}^{2N}$ with the usual symplectic structure, cannot be symplectically embedded into $B^2(R)\times \bb{R}^{2N-2}$
unless $r\leq R$. This motivated the definition of the invariant called the Gromov width.
Consider the ball of capacity $a$
$$ B^{2N}_a = \Big \{ z \in \bb{C}^N \ \Big | \ \pi \sum_{i=1}^N |z_i|^2 < a \Big \} , $$
with the standard symplectic form
$\omega_{std} = \sum dx_j \wedge dy_j$.
The \textbf{Gromov width} of a $2N$-dimensional symplectic manifold $(M,\omega)$
is the supremum of the set of $a$'s such that $B^{2N}_a$ can be symplectically
embedded in $(M,\omega)$. 

In this paper we consider coadjoint orbits of the special orthogonal group. Let $G=SO(2n+1)$ or $G=SO(2n)$. 
Then the Lie algebra $\lie{g}$ is the vector space of skew symmetric matrices of appropriate size.
We will identify the Lie algebra dual $\lie{g}^*$ with $\lie{g}$ using the $G$ invariant pairing in $\lie{g}$,
$(A,B)=-\frac 1 2 \textrm{trace}(AB).$
Throughout the paper we use the notation
\begin{displaymath}R(\alpha)=
\left(\begin{array}{cc}
    \cos(\alpha)&  -\sin(\alpha)\\
\sin(\alpha)& \cos(\alpha)
    \end{array}
\right),\,\,\,
L(a)=
\left(\begin{array}{cc}
    0&  -a\\
a& 0
    \end{array}
\right)
\end{displaymath} 
We make the following choices of maximal tori
\scriptsize\begin{displaymath} T_{SO(2n+1)}=\left\{ \left( \begin{array}{ccccc}
R(\alpha_1) &&&&\\
&R(\alpha_2)&&&\\
&& \ddots &&\\
&&& R(\alpha_n)&\\
&&&&1
\end{array}\right) \right\},\,\,\,T_{SO(2n)}=\left\{\left( \begin{array}{cccc}
R(\alpha_1) &&&\\
&R(\alpha_2)&&\\
&& \ddots &\\
&&& R(\alpha_n)
\end{array}\right)\right\} \end{displaymath}
\normalsize where $\alpha_j\in S^1.$  The corresponding Lie algebra duals are \scriptsize\begin{displaymath} \lie{t}_{SO(2n+1)}^*=\left\{ \left( \begin{array}{ccccc}
L(a_1) &&&&\\
&L(a_2)&&&\\
&& \ddots &&\\
&&& L(a_n)&\\
&&&&0
\end{array}\right) \right\},\,\,\,\lie{t}_{SO(2n)}^*=\left\{\left( \begin{array}{cccc}
L(a_1) &&&\\
&L(a_2)&&\\
&& \ddots &\\
&&& L(a_n)
\end{array}\right)\right\} \end{displaymath}
\normalsize 
and we choose the positive Weyl chambers to consist of matrices with 
 $a_1\geq  a_2\geq a_3\geq \ldots\geq a_n$ in the case $G=SO(2n+1)$, and $a_1\geq a_2\geq a_3\geq \ldots\geq a_{n-1}\geq|a_n|$ in the case $G=SO(2n)$. 
We are using the convention that the exponential map $exp:\lie{t}_{SO(2)} \rightarrow T_{SO(2)}$ is
given by $L(a)\rightarrow R(2\pi a),$ that is $S^1\cong \bb{R}/\bb{Z}$. 
A point $\lambda \in \lie{g}^*$ and a coadjoint orbit through it are called \textbf{regular} if the stabilizer of $\lambda$ under coadjoint action is the maximal torus. Coadjoint orbits are in bijection with points in the positive Weyl chamber. Under this bijection, the regular points correspond to the interior of the chamber.
Fix a point $\lambda=(\lambda_1, \lambda_2, \lambda_3, \ldots, \lambda_n)$, in the interior of the positive Weyl chamber, $\lie{t}^*_+$,\footnotesize
\begin{displaymath}
 \lambda=\begin{cases}
 \,\,\,\,\left( \begin{array}{cccccc}
L(\lambda_1)& &&&&\\
&L(\lambda_2)&&&&\\
&&L(\lambda_3)&&&\\
&&& \ddots &&\\
&&&& L(\lambda_n)&\\
&&&&&0
\end{array}
\right)\in \lie{t}_{SO(2n+1)}^* & \textrm{ if }G=SO(2n+1)\\
&\\
\,\,\,\,\left( \begin{array}{ccccc}
L(\lambda_1)& &&&\\
&L(\lambda_2)&&&\\
&&L(\lambda_3)&&\\
&&& \ddots &\\
&&&& L(\lambda_n)
\end{array}
\right)\in \lie{t}_{SO(2n)}^* & \textrm{ if }G=SO(2n)
\end{cases}
\end{displaymath}\normalsize
Denote the orbit of the coadjoint action of $G$ on $\lambda$ by $\cc{O}_{\lambda}$. The orbit is also a symplectic manifold, with Kostant-Kirillov symplectic form. The dimension of $\cc{O}_{\lambda}$ is equal to
\begin{equation*}
 \dim \cc{O}_{\lambda}=\dim\,(\lie{g}^*)-\dim T_G=
\begin{cases}
n(2n+1)-n=2n^2 &\textrm{ if }G=SO(2n+1)\\
n(2n-1)-n=2n(n-1) &\textrm{ if }G=SO(2n).
\end{cases}
\end{equation*}

One of the fundamental invariants of symplectic manifolds is the Gromov width defined above.
The purpose of this paper is to calculate the Gromov width of the orbit $\cc{O}_{\lambda}$.
We find the value of this invariant for certain orbits by proving that the lower bound is equal to the upper bound established by Zoghi in \cite{Z}. 

Given a Hamiltonian torus action one can construct symplectic embeddings of balls using information from the momentum polytope. This method uses the theorem of Karshon and Tolman, \cite{KT1}, recalled here as Proposition \ref{embed}, as explained in Example \ref{centeredexample}.
Using this technique we prove the following theorem.
\begin{theorem}\label{main}
Consider the coadjoint orbit $M:=\mathcal{O}_{\lambda}$ of the special orthogonal group through a regular point $\lambda$.
The Gromov width of $M$ is at least the minimum 
$$\min\{\, \left|\left\langle \alpha^{\vee},\lambda \right\rangle \right|\,; \alpha^{\vee} \textrm{ a  coroot}\}.$$
\end{theorem} 
In the case of $G=SO(2n+1)$ this result can be strengthened to cover also a class of orbits that are not regular (see Section \ref{notregular}).
The analysis of the root system of the special orthogonal groups done in Subsection \ref{rootsystem}, and inequalities imposed on $\lambda$, imply that if $G=SO(2n+1)$ this minimum is equal to
$$\min\{\lambda_1-\lambda_{2},\ldots,\lambda_{n-1}-\lambda_{n},\,2\lambda_n\},$$ 
while for $G=SO(2n)$ the minimum is
$$\min\{\lambda_1-\lambda_{2},\lambda_2-\lambda_{3},\ldots,\lambda_{n-1}-\lambda_{n},\lambda_{n-1}+\lambda_{n}\}.$$
\indent There are reasons to care about this particular lower bound. 
Zoghi in \cite{Z} analyzed orbits satisfying some additional integrality conditions. He called an orbit $\cc{O}_{\lambda}$ {\bf indecomposable} if there exists a simple root $\alpha$ such that for each root $\alpha'$ there exists a positive integer $k$ (depending on $\alpha'$) such that
$$k\,\langle \alpha^{\vee},\lambda \rangle=\langle (\alpha')^{\vee},\lambda \rangle .$$ 
In particular monotone orbits are indecomposable. Zoghi proved that for compact connected simple Lie group $G$ the formula 
$\min\{\, \left|\left\langle \alpha^{\vee},\lambda \right\rangle \right|\,; \alpha^{\vee} \textrm{ a  coroot}\}$
gives an upper bound for Gromov width of regular indecomposable $G$-coadjoint orbit through $\lambda$ (\cite[Proposition 3.16]{Z}). 
Combining his theorem for $G=SO(n)$ with our Theorem \ref{main} we obtain
\begin{corollary}
 The Gromov width of a regular indecomposable coadjoint $SO(n)$ orbit $\mathcal{O}_{\lambda}$ is exactly the minimum $$\min\{\, \left|\langle \alpha^{\vee},\lambda \rangle \right|\,; \alpha^{\vee} \textrm{ a  coroot}\}.$$
\end{corollary}
 Zoghi also proved that the same formula gives the Gromov width for regular indecomposable $U(n)$ coadjoint orbits. Moreover, the author proved in \cite{P}, that for a class of non-regular $U(n)$ coadjoint orbits, the lower bound of Gromov width is given by minimum over non-zero elements of the above set, that is $$\min\{\, \left|\left\langle \alpha^{\vee},\lambda \right\rangle \right|\,; \alpha \textrm{ a  coroot and }\langle \alpha^{\vee},\lambda \rangle \neq 0\}.$$
The same formula describes the Gromov width of complex Grassmannians, a different class of non-regular $U(n)$ coadjoint orbits (\cite{KT1}). 

To prove Theorem \ref{main} we recall an action of the Gelfand-Tsetlin torus on an open dense subset of $\mathcal{O}_{\lambda}$. We then use
the theorem of Karshon and Tolman \cite{KT1} to obtain symplectic embeddings of balls. 
Coadjoint orbits come equipped with the Hamiltonian action of the maximal torus of the group. One can apply the Karshon and Tolman's result (Proposition \ref{embed}) to the region centered with respect to this standard action and obtain a lower bound for Gromov width of the orbit. This is how Zoghi proved in \cite{Z} the lower bounds of Gromov width of regular $U(n)$ coadjoint orbits.
If the root system is non-simply laced, the lower bound obtained this way is weaker (i.e. lower) then the lower bound we prove here. This phenomenon is explained in the Appendix \ref{standardnotenough}. In other words, the lower bounds for $SO(2n+1)$ we prove here could not be obtained using the standard action of maximal torus.
\\ \indent \textbf{Organization.} Section \ref{Preliminaries} contains preliminaries about the centered regions and root systems. In Section \ref{gtssystem} we describe the Gelfand Tsetlin system and an action it is inducing, while in Section \ref{polytope} we analyze the image of the momentum polytope. Section \ref{edges} is devoted to the computation of weights of this action. The proof of the Theorem \ref{main} is in Section \ref{proof}. Later, in Section \ref{notregular} we prove the generalization of the main theorem to the class of $SO(2n+1)$ orbits that are not regular. Appendix \ref{standardnotenough} explains why our result is so important for groups whose root system is non-simply laced. Second appendix, Appendix \ref{polytopeproof}, contains proofs of the lemmas used to analyze the Gelfand-Tsetlin polytope.
\\ \indent \textbf{Acknowledgments.} The author is very grateful to Yael Karshon for suggesting this problem and helpful conversations during my work on this project. 
The author also would like to thank her advisor, Tara Holm, for useful discussions.
\section{Preliminaries}\label{Preliminaries}
\subsection{Centered actions and a theorem of Karshon and Tolman}
Centered actions were introduced in \cite{KT2}. Here we briefly recall the definition and refer the reader to \cite{KT1} or \cite{P} for more explanation and examples.

Let $(M, \omega)$ be a connected symplectic manifold, equipped with an effective, symplectic action of a torus $T\cong (S^1)^{\dim T}.$
The action of $T$ is called \textbf{Hamiltonian} if there exists a $T$-invariant map
$\Phi \colon M \to \lie{t}^*$, called the \textbf{momentum map}, such that
\begin{equation} \label{def moment}
        \iota(\xi_M) \omega =  d \left< \Phi,\xi \right>
\quad \forall \ \xi \in \lie{t},
\end{equation}
where $\xi_M$ is the vector field on $M$ generated by $\xi \in \lie{t}$. Note that with our sign convention the isotropy weights of $T$ action on $T_pM$, where $p$ is a fixed point, are pointing out of the momentum map image. Let $\cc{T} \subset \lie{t}^*$ be an open convex set which contains $\Phi(M)$. The quadruple $(M,\omega,\Phi,\cc{T})$ is a
\textbf{proper Hamiltonian $\mathbf{T}$-manifold} if 
$\Phi$ is proper as a map to $\cc{T}$, that is,
the preimage of every compact subset of $\cc{T}$ is compact.
For any subgroup $K$ of $T$, let
$M^K = \{ m \in M \mid a \cdot m = m \ \forall a \in K \}$
denote its fixed point set.

\begin{definition} \label{centered-definition}
A proper Hamiltonian $T$-manifold $(M,\omega,\Phi,\cc{T})$
is \textbf{centered} about a point $\alpha \in \cc{T}$ if
$\alpha$ is contained in the momentum map image of every component
of $M^K$, for every subgroup $K \subseteq T$.
\end{definition}

\begin{proposition} (Karshon, Tolman, \cite{KT1})\label{embed}
Let $(M,\omega,\Phi,\cc{T})$ be a proper Hamiltonian $T$-manifold.
Assume that $M$ is centered about $\alpha \in \cc{T}$ and that 
$\Phi^{-1}(\{\alpha\})$ consists of a single fixed point $p$.
Then $M$ is equivariantly symplectomorphic to 
$$\left\{ z \in \bb{C}^n \ | \ \alpha + \pi \sum |z_j|^2 \eta_j \in \cc{T} \right\},$$
where  $-\eta_1,\ldots,-\eta_n$ are the isotropy weights at $p$.
\end{proposition}
Note that the above formlumation differs from the one in \cite{KT1} by a minus sign. This is due to the fact that our definition of momentum map (\ref{def moment}) also differs by a minus sign from the definition used in \cite{KT1}.
\begin{example}\label{centeredexample}
Consider a compact symplectic toric manifold $M$ whose momentum map image is the closure of the following region.
\begin{center}
\includegraphics[width=0.6\textwidth]{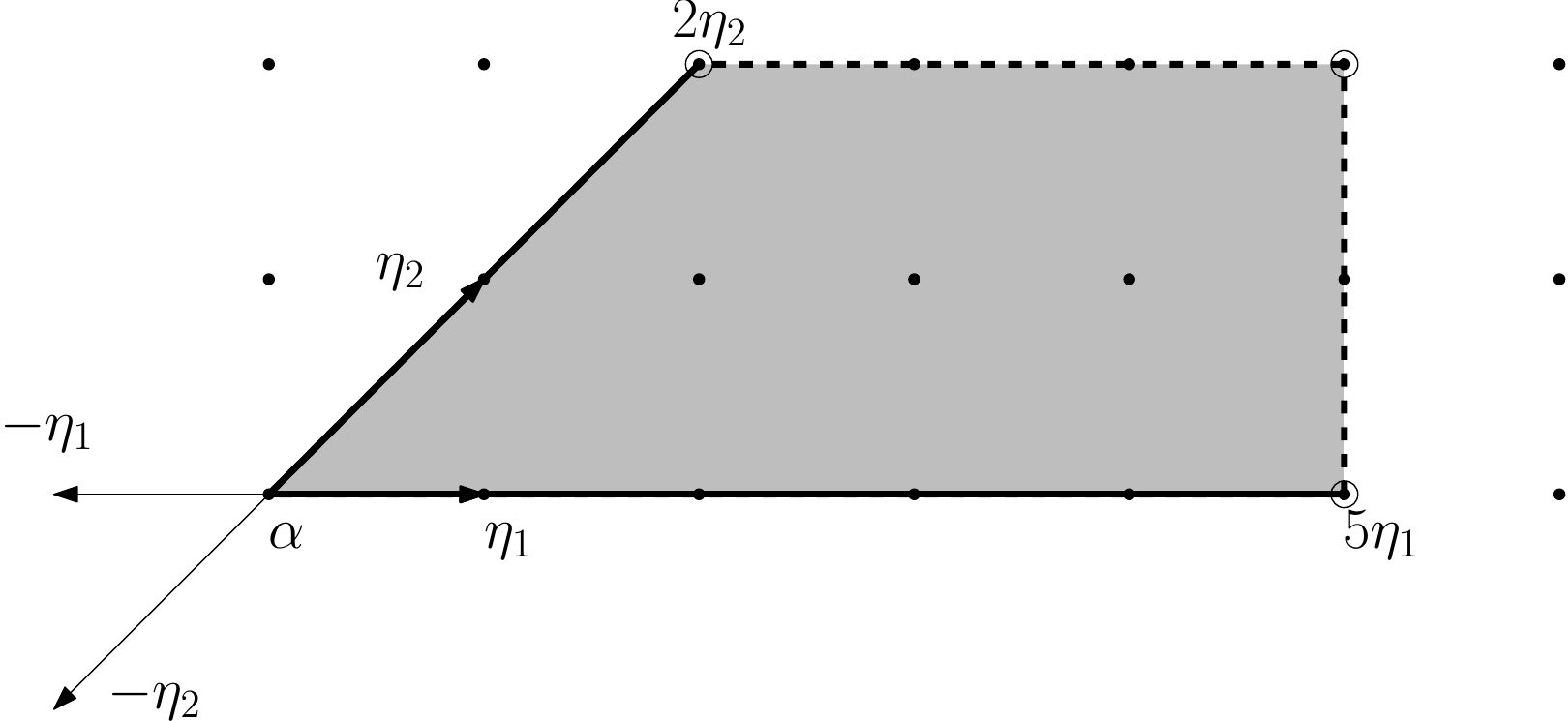}
\end{center}
The weights of the torus action are $(-\eta_1)$ and $(-\eta_2)$, and the lattice lengths of edges starting from $\alpha$ are $5$ and $2$
(with respect to the weight lattice).
The largest subset of $M$ that is centered about $\alpha$ maps under the momentum map to the shaded region. 
The above Proposition tells us that this centered region is
equivariantly symplectomorphic to
$$\{z \in \bb{C}^2| \alpha + \pi (|z_1|^2 +|z_2|^2) \in \textrm{ shaded region } \}.$$
If $z\in B^{4}_2 =  \{ z \in \bb{C}^2 \ \Big |  \pi (|z_1|^2  +|z_2|^2 ) <2\}$
then $\alpha + \pi (|z_1|^2 \eta_1 +|z_2|^2 \eta_2)$ is in the shaded region.
Therefore the $4$-dimensional ball $B^{4}_2$ of capacity $2$ embeds into $M$ and the Gromov width of $M$ is at least the minimum of lattice lengths of edges of the moment polytope, starting at $\alpha$.
 \end{example}
\subsection{Root system of the special orthogonal group.}\label{rootsystem}
The root system of a group $G$ consists of vectors in $\lie{t}^*$, the dual of the Lie algebra of the maximal torus of $G$.
The coroot $\alpha^{\vee}$ corresponding to a root $\alpha$ is an element of $\lie{t}$ given by the condition
$x(\alpha^{\vee})=2\,\frac{ \left\langle \alpha, x\right\rangle}{ \left\langle \alpha, \alpha \right\rangle}$
for all $x \in \lie{t}^*$. Recall that $x(\alpha^{\vee})=- \frac 1 2 \textrm{trace}(x\,\alpha^{\vee}).$
We will often denote this pairing between $\lie{t}$ and $\lie{t}^*$ by $\left\langle,  \right\rangle.$ We identify $\lie{t}^*$ (so also $\lie{t}$) with $\bb{R}^n$ by sending matrices \scriptsize\begin{displaymath} \left( \begin{array}{ccccc}
L(a_1) &&&&\\
&L(a_2)&&&\\
&& \ddots &&\\
&&& L(a_n)&\\
&&&&0
\end{array}\right) \in \lie{t}_{SO(2n+1)}^*,\,\,\,\left( \begin{array}{cccc}
L(a_1) &&&\\
&L(a_2)&&\\
&& \ddots &\\
&&& L(a_n)
\end{array}\right) \in \lie{t}_{SO(2n)}^*\end{displaymath}
\normalsize to $(a_1,a_2,\ldots,a_n)\in \bb{R}^n$. With this identification, the pairing $\left\langle,  \right\rangle$ in $\lie{t}^*$ is just the standard scalar product.

The root system of the group $SO(2n+1)$ consists of vectors $=\pm e_j$, $j=1, \ldots n$, of squared length $1$, and vectors $\pm (e_j\pm e_k)$, $j\neq k$, of squared length $2$ in the Lie algebra dual $\lie{t}_{SO(2n+1)}^*$. 
Therefore this root system for $SO(n)$ is non-simply laced. 
Note that $$\left\langle  (e_j\pm e_k)^{\vee} , \lambda \right\rangle= 2 \frac{\left\langle e_j\pm e_k,\lambda\right\rangle}{\left\langle e_j\pm e_k,e_j\pm e_k\right\rangle}=\lambda_j \pm \lambda_k$$ and 
$$\left\langle (e_j)^{\vee} , \lambda \right\rangle=2 \frac{\left\langle e_j,\lambda\right\rangle}{\left\langle e_j,e_j\right\rangle}=2\lambda_j.$$
Therefore for $\lambda$ in our chosen positive Weyl chamber
$$\min\{ \left|\left\langle \alpha^{\vee},\lambda \right\rangle \right|\,; \alpha^{\vee} \textrm{ a  coroot}\}=\min\{\lambda_1-\lambda_{2},\ldots,\lambda_{n-1}-\lambda_{n},\,2\lambda_n\}.$$

The root system for $SO(2n)$ is simply laced and consists of vectors  $\pm (e_j\pm e_k)$, $j\neq k$, of squared length $2$. 
Note that $$\left\langle  (e_j\pm e_k)^{\vee} , \lambda \right\rangle= 2 \frac{\left\langle e_j\pm e_k,\lambda \right\rangle}{\left\langle e_j\pm e_k,e_j\pm e_k\right\rangle}=\lambda_j \pm \lambda_k.$$
Therefore for $\lambda$ in a positive Weyl chamber
$$\min\{ \left|\left\langle \alpha^{\vee},\lambda \right\rangle \right|\,; \alpha^{\vee} \textrm{ a  coroot}\}=\min\{\lambda_1-\lambda_{2},\lambda_2-\lambda_{3},\ldots,\lambda_{n-1}-\lambda_{n},\lambda_{n-1}+\lambda_{n}\}.$$
\section{The Gelfand-Tsetlin system.}\label{gtssystem}
In this section we describe the Gelfand-Tsetlin (sometimes spelled Gelfand-Cetlin, or Gelfand-Zetlin) system of action coordinates, which originally appeared in \cite{GS1}. 
Consider the following sequence of subgroups
\begin{equation*}
 G_{n}=SO(n) \supset G_{n-1}=SO(n-1) \supset G_{n-2}=SO(n-2) \supset \ldots \supset G_2=SO(2).\end{equation*}
For these groups we make the following choices of maximal tori.
\scriptsize\begin{displaymath} T_{SO(2k+1)}=\left( \begin{array}{ccccc}
R(\alpha_1) &&&&\\
&R(\alpha_2)&&&\\
&& \ddots &&\\
&&& R(\alpha_k)&\\
&&&&1
\end{array}\right),\,\,\,\,\,T_{SO(2k)}=\left( \begin{array}{cccc}
R(\alpha_1)& &&\\
&R(\alpha_2)&&\\
&& \ddots &\\
&&& R(\alpha_k)\end{array}\right).\end{displaymath}
\normalsize The positive Weyl chambers are chosen in an analogous way to the case described in the Introduction. Take any $G_k$ from this sequence, $k=2, \ldots, 2n$.
The group $G_k$ injects into $G$ by 
$$G_k \ni B \mapsto \left(\begin{array}{c|c}
B & 0\\
\hline
0& I\end{array}\right).$$
Therefore it also act on $\cc{O}_{\lambda}$ by a subaction of the coadjoint action.
This action is Hamiltonian with a momentum map $\Phi^k :\cc{O}_{\lambda} \rightarrow \lie{so}(k)^*$ sending a matrix $A=[a_{ij}]$ to the 
$k \times k$ top left submatrix of $A$, which we denote by $\Phi^k(A)$ or $(A)_k$ for short. 
The action of the Gelfand-Tsetlin torus is 
defined using the following functions. Compose the map $\Phi^k$ with the map $s_k : \lie{so}(k)^* \rightarrow (\lie{t}_{SO(k)})^*_+$ sending $A \in \lie{so}(k)^*$ to the unique point of intersection of the $SO(k)$-orbit, $SO(k)\,\cdot\,A$, with the positive Weyl chamber. Recall that we identify Lie algebra dual $(\lie{t}_{SO(k)})^*$ with $\bb{R}^{\lfloor\frac{k}{2}\rfloor}$, as explained in the previous section.
The positive Weyl chamber, $(\lie{t}_{SO(k)})^*_+$, is identified with the subset of points $(x_1, \ldots, x_{\lfloor\frac{k}{2}\rfloor}) \in \bb{R}^{\lfloor\frac{k}{2}\rfloor}$ satisfying $x_1 \geq x_2 \geq \ldots \geq x_{\lfloor\frac{k}{2}\rfloor},$ for $k$ odd, and  $x_1 \geq x_2 \geq \ldots \geq x_{\frac{k}{2}-1}\geq |x_{\frac{k}{2}}|,$ for $k$ even.

The composition $s_k \circ \Phi^k:\cc{O}_{\lambda} \rightarrow (\lie{t}_{SO(k)})^*_+$ gives us $\lfloor\frac{k}{2}\rfloor$ continuous (not everywhere smooth) functions which we denote $$\Lambda^{(k)}:=(\lambda^{(k)}_{1},\ldots , \lambda^{(k)}_{\lfloor\frac{k}{2}\rfloor}).$$ In this notation the superscript keeps track of the dimension of the matrices in the group (not the dimension of the maximal torus). Note that due to our choices of positive Weyl chambers, the only Gelfand-Tsetlin functions that can be negative are $\{x^{(k)}_{\frac k 2}\}$, for $k$ even.
\begin{displaymath}
\xymatrix{
\mathcal{O}_{\lambda} \ar[r]^{\Phi^k} \ar[rd]_{\Lambda^{(k)}} &
\lie{so}(k)^*  \ar[d]^{s_k}&
 \\
&(\lie{t}_{SO(k)})^*_+
} 
\end{displaymath}
These functions are related to the following action of $T_{SO(k)}$ denoted by $*$. An element $t \in T_{SO(k)}$ acts on a point $A \in \cc{O}_{\lambda}$ by the standard $SO(k)$ action of $B^{-1}\, t\, B$, 
where $B \in SO(k)$ is such that 
$B\,\Phi^{k}(A)\, B^{-1} \in (\lie{t}_{SO(k)})^*_+$:
$$t*A:= \left(\begin{array}{c|c}B^{-1}\, t\, B& \\ \hline &I_{n-k}\end{array}\right)\,A\,\left(\begin{array}{c|c}B^{-1}\, t\, B& \\ \hline &I_{n-k}\end{array}\right)\inv.$$

Similarly to the unitary case, one can show 
\begin{proposition}
The function $\Lambda^{(k)}$ is smooth at the preimage of the interior of the positive Weyl chamber, $$U_{SO(k)}:=(\Lambda^{(k)})\inv (\textrm{int }(\lie{t}_{SO(k)})^*_+).$$
Moreover, the $*$ action of the torus $T_{SO(k)}$ on $U_{SO(k)}$ is Hamiltonian and $\Lambda^{(k)}$ is a momentum map.
\end{proposition}\begin{proof}
The proofs are analogous to the unitary case, described in \cite[Proposition 3.2]{P} and \cite[Proposition 3.4]{P}. 
\end{proof}
If $G=SO(2n+1)$, putting together these functions for $k=1, \ldots 2n$ we obtain a function, denoted by $\Lambda=\{\lambda^{(k)}_j\,| 1\leq k \leq 2n, \, 1\leq j \leq \lfloor\frac{k}{2}\rfloor\}$, mapping $\cc{O}_{\lambda}$ to $\bb{R}^N$, where $$N= n+ 2(n-1) +2(n-2)+\ldots +2 \cdot 2= n + n(n-1)=n^2.$$
If $G=SO(2n)$, then we obtain a function $\Lambda=\{\lambda^{(k)}_j\,| 1\leq k \leq 2n-1, \, 1\leq j \leq \lfloor\frac{k}{2}\rfloor\}$, mapping $\cc{O}_{\lambda}$ to $\bb{R}^N$, with $$N=2(n-1) +2(n-2)+\ldots+ 2 \cdot 2= n(n-1).$$
In both cases $N$ is equal to half of the dimension of a regular coadjoint orbit of $G$.

 Putting the actions together we obtain the Hamiltionian action of the {\bf Gelfand-Tsetlin torus} $T=T_{GT}=T_{SO(n-1)}\oplus \ldots \oplus T_{SO(2)}\cong(S^1)^N$ on the dense open subset $$ U:=\cap_k \, U_{SO(k)}$$ of the coadjoint orbit $\cc{O}_{\lambda}$ where all functions $\Lambda^{(k)}$ are smooth.
 This action is called the {\bf Gelfand-Tsetlin action} and its momentum map is $\Lambda$.

\section{The Gelfand-Tsetlin polytope}\label{polytope}
In this section we describe in details the image of Gelfand-Tsetlin functions, $\Lambda(\cc{O}_{\lambda})$. The fact that the image forms a polytope seems to be well known. However we could not find a reference for this fact. Therefore we prove it below.
The following lemmas are helpful in analyzing the image of Gelfand-Tsetlin functions. 
Their proofs are in the Appendix \ref{polytopeproof}.
\begin{lemma}\label{openodd} For any real numbers \begin{equation} \label{ineqodd} b_1 \geq a_1 \geq b_2 \geq a_2 \geq \ldots \geq a_{k-1} \geq b_{k} \geq |a_{k}|\end{equation}
 there exist a real vector $Y=[y_1, \ldots, y_{2k}]^T$ in $\bb{R}$ 
such that
the skew symmetric matrices 
\small
\begin{displaymath}
A:=\left(\begin{array}{c|c}
\begin{array}{cccc}
  L(a_1) &  & &\\
 &L(a_2)   & &\\
& &\ddots & \\
&&& L(a_k)
    \end{array} & Y\\
\hline
-Y^T &0 
\end{array}
\right)\textrm{  and  }S:=\left(\begin{array}{c|c}
\begin{array}{cccc}
  L(b_1) &  & &\\
 &L(b_2)   & &\\
& &\ddots & \\
&&& L(b_k)
    \end{array} & 0\\
\hline
0 &0 
\end{array}
\right).
\end{displaymath} \normalsize are in the same $SO(2k+1)$ orbit. Moreover, \\
(1) if $a_j,b_j$ are not satisfying inequalities (\ref{ineqodd}), then such $Y$ does not exist,\\
(2) if $j$ is the unique index from $1,\ldots, k$ such that $a_j=b_m$ for some $m$, then $y_{2j-1}=y_{2j}=0$. \end{lemma}
Here is the even dimensional analogue.
\begin{lemma} \label{openeven}For any real numbers \begin{equation}\label{ineqeven}
 a_1 \geq b_1 \geq a_2 \geq b_2\geq \ldots \geq b_{k-1} \geq |a_{k}|\end{equation} there exist a real vector $Y=[y_1, \ldots, y_{2k-1}]^T$ in $\bb{R}$ 
such that
the skew symmetric matrices 
\begin{displaymath}
A:=\left(\begin{array}{c|c}
\begin{array}{ccccc}
  L(b_1) &  && &\\
 &L(b_2)   && &\\
& &\ddots & &\\
&&& L(b_{k-1})&\\&&&&0
    \end{array} & Y\\
\hline
-Y^T &0 
\end{array}
\right)\textrm{  and  }\left(
\begin{array}{cccc}
  L(a_1) &  & &\\
 &L(a_2)   & &\\
& &\ddots & \\
&&& L(a_k)
    \end{array} 
\right).
\end{displaymath} are in the same $SO(2k)$ orbit. Moreover, \\
(1) if $a_j,b_j$ are not satisfying inequalities (\ref{ineqeven}), then such $Y$ does not exist,\\
(2) if $j$ is the unique index from $1,\ldots, k$ such that $b_j=a_m$ for some $m$, then $y_{2j-1}=y_{2j}=0$. \end{lemma}
\subsection{The polytope for $SO(2n+1)$.}
Now we are ready to describe the image of the Gelfand-Tsetlin functions for the case $G=SO(2n+1)$, in $\bb{R}^{n^2}$.
Let $\{x^{(k)}_j\,| 1\leq k \leq 2n, \, 1\leq j \leq \lfloor \frac k 2 \rfloor \}$ be basis of $\bb{R}^{n^2}$.
\begin{proposition}\label{polytopeodd}
 For $SO(2n+1)$ the image of the Gelfand-Tsetlin functions $\Lambda: \cc{O}_{\lambda} \rightarrow \bb{R}^{n^2}$ 
is the polytope, which we will denote by $\cc{P}$, defined by the following set of inequalities
\begin{equation}\label{polytopeineqodd}
\begin{cases}
\,\,\,x^{(2k)}_{1} \geq x^{(2k-1)}_{1} \geq x^{(2k)}_{2} \geq x^{(2k-1)}_{2} \geq \ldots \geq x^{(2k)}_{k-1} \geq x^{(2k-1)}_{k-1} \geq |x^{(2k)}_{k}| ,\\
\,\,\,x^{(2k+1)}_{1} \geq x^{(2k)}_{1} \geq x^{(2k+1)}_{2} \geq x^{(2k)}_{2} \geq \ldots \geq x^{(2k+1)}_{k} \geq |x^{(2k)}_{k}| ,\\
\end{cases}
\end{equation}
for all $k=1,\ldots,n$, where $x^{(2n+1)}_j=\lambda_j$. 
\end{proposition}
\begin{proof}
The above proposition follows from consecutive applications of Propositions \ref{openodd} and \ref{openeven}.
We will show only the first two steps as the next ones are analogous.
(Similar procedure for the unitary case is described in the proof of Proposition 3.5 in \cite{P}.)\\
\indent Take any sequence of numbers $\{x^{(l)}_j\}$ satisfying inequalities (\ref{polytopeineqodd}). Lemma \ref{openodd} implies that there exist a real vector $Y_1$ 
such that the matrix 
\small
\begin{displaymath}
A_1:=\left(\begin{array}{c|c}
\begin{array}{cccc}
  L(x^{(2n)}_{1}) &  & &\\
 &L(x^{(2n)}_{2})   & &\\
& &\ddots & \\
&&& L(x^{(2n)}_{n})
    \end{array} & Y_1\\
\hline
-Y_1^T &0 
\end{array}
\right)\end{displaymath}
\normalsize is in the same $SO(2k+1)$ orbit as $\lambda$, i.e. $B_1\,A_1\,B_1\inv=\lambda$ for some matrix $B_1 \in SO(2n+1)$.
Now we apply Lemma \ref{openeven} to find a real vector $Y_2$ and a matrix $B_2 \in SO(2n)$ such that for the matrix
\small
\begin{displaymath}
A_2:=\left(\begin{array}{c|c}
\begin{array}{ccccc}
  L(x^{(2n-1)}_{1}) &  && &\\
 &L(x^{(2n-1)}_{2})   && &\\
& &\ddots & &\\
&&& L(x^{(2n-1)}_{n-1})&\\&&&&0
    \end{array} & Y_2\\
\hline
-Y_2^T &0 
\end{array}
\right)\end{displaymath}
\normalsize we have
$$B_2 A_2 B_2 \inv =\left(\begin{array}{cccc}
  L(x^{(2n)}_{1}) &  & &\\
 &L(x^{(2n)}_{2})   & &\\
& &\ddots & \\
&&& L(x^{(2n)}_{n})
    \end{array} \right).$$
Therefore the matrix 
\begin{displaymath}
\left(\begin{array}{c|c}
A_2 & B_2\inv Y_1\\
\hline
-Y_1^T B_2  &0 
\end{array}
\right)\end{displaymath}
has desired values of the Gelfand-Tsetlin functions $x^{(2n)}_*,x^{(2n-1)}_*$ and is in $\cc{O}_{\lambda}$ as
\begin{displaymath}B_1\,\left(\begin{array}{c|c}
B_2 & \\
\hline
 &1
\end{array}
\right)\,
\left(\begin{array}{c|c}
A_2 & B_2\inv Y_1\\
\hline
-Y_1^T B_2  &0 
\end{array}
\right)\,\left(\begin{array}{c|c}
B_2\inv & \\
\hline
 &1
\end{array}
\right)\,B_1\inv\end{displaymath}
\begin{displaymath}=B_1\,\left(\begin{array}{c|c}
B_2 A_2 B_2 \inv &  Y_1\\
\hline
-Y_1^T  &0 
\end{array}
\right)\,B_1 \inv =B_1\, A_1 \,B_1 \inv=\lambda. \end{displaymath}
Succesively repeating similar steps, one can construct a matrix in $\cc{O}_{\lambda}$ with prescribed values of Gelfand-Tsetlin 
functions if only these values satisfy inequalities
(\ref{polytopeineqodd}).
\end{proof}
We can think of the Gelfand-Tsetlin polytope as the set of points whose coordinates fit into the following triangle of inequalities. 
Let the first row be given by $\lambda_1, \ldots, \lambda_n$ (or $|\lambda_n|$ in $SO(2n)$ case). Form next rows from the coordinates with the same superscript so that top left and right left neighbors of the coordinate $x^{(k)}_{j}$ are $x^{(k+1)}_{j}$ and $x^{(k+1)}_{j+1}$. The value of $x^{(k)}_{j}$ must be between the values of its top left and top right neighbors. \begin{center}\label{triangleofinequalities}
\includegraphics{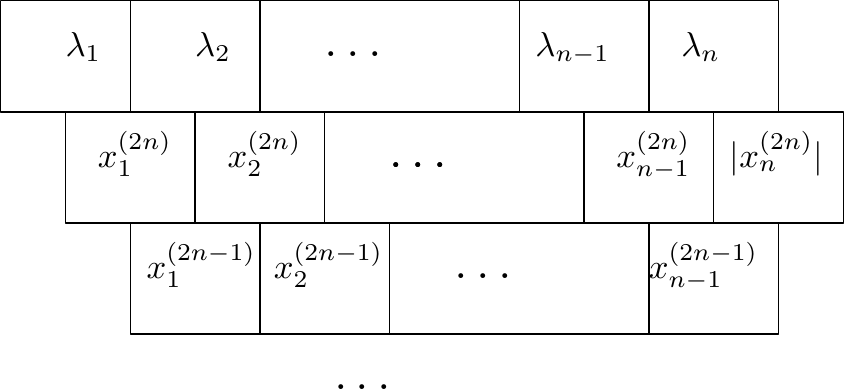}
\end{center}
\subsection{The polytope for $SO(2n)$.} 
Situation for $G=SO(2n)$ is very similar. Let $\{x^{(k)}_j\,| 1\leq k \leq 2n-1, \, 1\leq j \leq \lfloor \frac k 2 \rfloor \}$ be basis of $\bb{R}^N=\bb{R}^{n(n-1)}$.
\begin{proposition}\label{polytopeeven}
 For $SO(2n)$ the image of the Gelfand-Tsetlin functions $\Lambda: \cc{O}_{\lambda} \rightarrow \bb{R}^{n(n-1)}$ 
is the polytope, which we will denote by $\cc{P}$, defined by the following set of inequalities
\begin{equation}\label{polytopeineqeven}
\begin{cases}
\,\,\,x^{(2k)}_{1} \geq x^{(2k-1)}_{1} \geq x^{(2k)}_{2} \geq x^{(2k-1)}_{2} \geq \ldots \geq x^{(2k)}_{k-1} \geq x^{(2k-1)}_{k-1} \geq |x^{(2k)}_{k}| ,\\
\,\,\,x^{(2k+1)}_{1} \geq x^{(2k)}_{1} \geq x^{(2k+1)}_{2} \geq x^{(2k)}_{2} \geq \ldots \geq x^{(2k+1)}_{k} \geq |x^{(2k)}_{k}| ,\\
\end{cases}
\end{equation}
for all $k=1,\ldots,n$, where $x^{(2n)}_j=\lambda_j$ for $j=1,\ldots, n$. 
\end{proposition}
\begin{proof}
 Analogous to the proof of Proposition \ref{polytopeodd}.
\end{proof}
Here we also can present these inequalities in the form of a triangle of inequalities similar to the $SO(2n+1)$ case above.

\section{Isotropy weights of the Gelfand-Tsetlin action}\label{edges}
Notice that $\Lambda(\lambda)$ is a vertex of $\cc{P}$. This is because at this point all the Gelfand-Tsetlin functions are equal to their upper bounds. If on the triangle of inequalities we connect by a line all coordinates of $\Lambda(\lambda)$ with the same values, then we obtain the picture in Figure \ref{triangleforvertexlambda}.
\begin{figure}[h]\label{triangleforvertexlambda}
 \includegraphics[width=.3\textwidth]{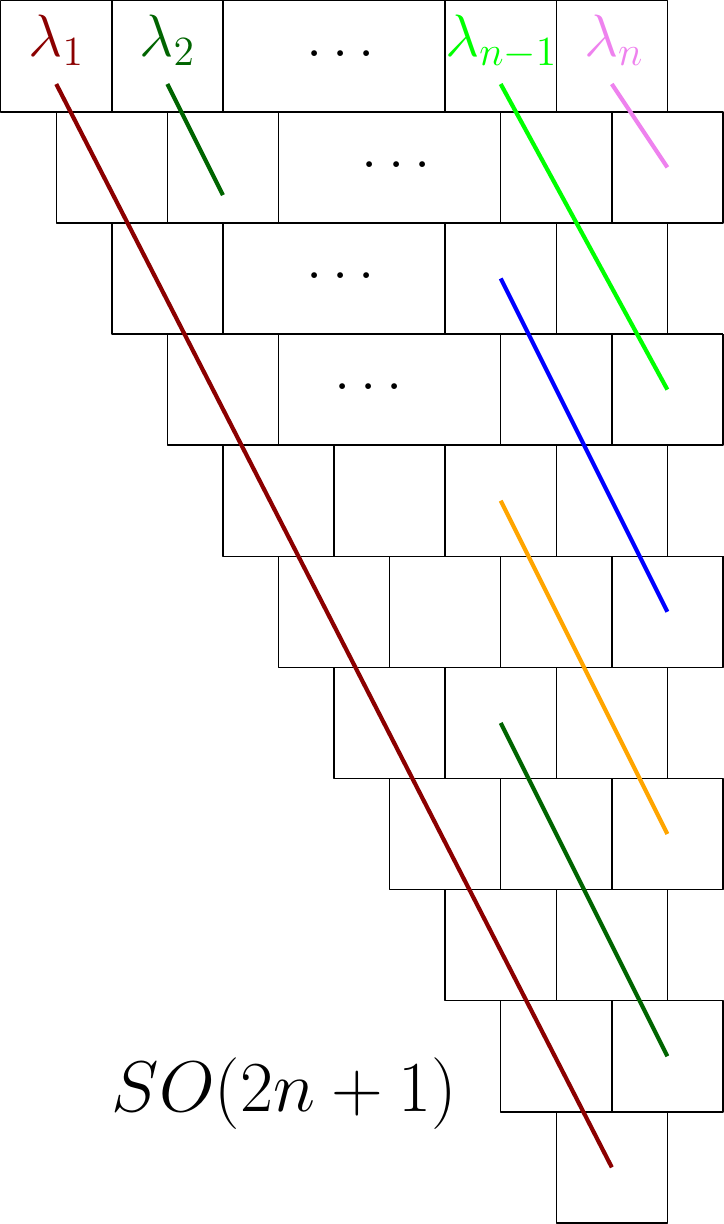} \hspace{20mm}
\includegraphics[width=.3\textwidth]{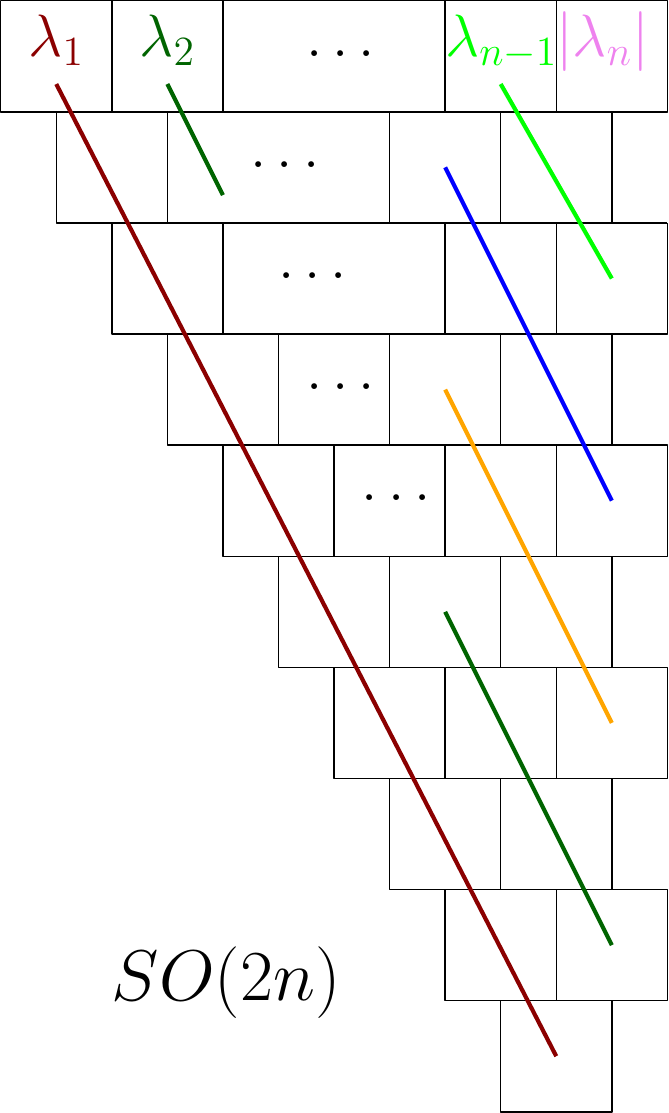}
\caption{Triangles of inequalities for $\Lambda(\lambda)$ in $G=SO(2n+1)$ and $G=SO(2n)$ cases.}
\label{triangleforvertexlambda}
\end{figure}

We will analyze edges starting from $\Lambda(\lambda)$. For more details about identifying vertices and edges of the Gelfand-Tsetlin polytope, see 
Lemmas 3.7 and 3.10 in \cite{P} or \cite{Zi}. Basically, to obtain an edge starting from $\Lambda(\lambda)$, we pick one of the inequalities defining $\cc{P}$ that are equations at $\Lambda(\lambda)$, and consider the set of points in $\cc{P}$ satisfying all the same equations that $\Lambda(\lambda)$ satisfies, except possibly this chosen one. It is important to note that in this way we obtain ALL the edges starting from $\Lambda(\lambda)$. This procedure may not work if instead of $\Lambda(\lambda)$ we analyze a vertex $V'$ of $\cc{P}$ such that $\Lambda\inv (V')$ is not in a subset of $U$. 

Pick any $k \in \{1, \dots ,n\} $ for $G=SO(2n+1)$, or $k \in \{1, \dots ,n-1\} $ for $G=SO(2n)$, and $j \in \{1, \dots k\}$. Consider the set $E:=E^{(2k)}_{j}$, that is the image of points where all the Gelfand-Tsetlin functions are equal to their upper bound, apart from the function $\lambda^{(2k)}_{j}$. That is, $E$ is the line segment consisting of points ${\bf x} \in \bb{R}^N$ satisfying
 \begin{align}\label{edgeeven}
x^{(m)}_{l}&=\lambda_l\textrm{ for all }m\textrm{ and for all }l\neq j,\notag\\
x^{(m)}_{j}&=\lambda_j\textrm{ for all }m > 2k,\notag \\
x^{(m)}_{j}&=x^{(2k)}_{j} \textrm{ for all }2j \leq m \leq 2k,\\
x^{(2k)}_{j}&\in [ \lambda_{j+1},\lambda_j ] \textrm{ if }j < k,\notag\\
x^{(2k)}_{j}&\in [-\lambda_{k},\lambda_k ] \textrm{ if }j=k.\notag
 \end{align}
 The following graphical presentation (of the case $j<k$) can be helpful.
\begin{equation*}
  \begin{array}{ccccccccccccc}
 \lambda_{j-1}&&&\lambda_{j}&&&\lambda_{j+1}&&&&&&\\
&\begin{rotate}{135}=  \end{rotate} &&&\begin{rotate}{135} = \end{rotate}&&&\begin{rotate}{135} = \end{rotate}&&&&\\
&& x^{(2k+1)}_{j-1}&&&x^{(2k+1)}_{j}&&&x^{(2k+1)}_{j+1}&&&\\
&&& \begin{rotate}{135} = \end{rotate} &&&\begin{rotate}{-45} $>$ \end{rotate}&&&\begin{rotate}{135} = \end{rotate}&&\\
&&&& x^{(2k)}_{j-1}&&&x^{(2k)}_{j}&&&x^{(2k)}_{j+1}&\\
&&&&&\begin{rotate}{135} = \end{rotate} &&&\begin{rotate}{135} = \end{rotate}&&&\begin{rotate}{135} = \end{rotate}\\
&&&&&& x^{(2k-1)}_{j-1}&&&x^{(2k-1)}_{j}&&&x^{(2k-1)}_{j+1}\\
 \end{array}
\end{equation*}
\normalsize
The set $E$ is an edge of $\cc{P}$. Proof of this fact in nearly identical as in the unitary case, described in Lemmas 3.7 and 3.10 of \cite{P}. The vertex $\Lambda(\lambda)$ belongs to $E$. Denote by $\eo{E}$ the half open line segment: $E$ minus the other endpoint, i.e. $\eo{E}=\Lambda(\lambda)\cup \textrm{ int }E$. 
From the definition of $U$ it follows that if $q \in U$ and $\Lambda(q)$ belongs to a face $\cc{F}$ of the polytope $\cc{P}$, then $\Lambda \inv (\textrm{int }\cc{F})$ is in $U$.
Therefore $\Lambda \inv (\eo{E})$ is also contained in $U$ and is equipped with a smooth action of the Gelfand-Tsetlin torus. Below we analyze carefully which matrices are in $\Lambda \inv (\eo{E})$.
\begin{lemma}\label{spherea}
 $\Lambda \inv (\eo{E})$ is a disc invariant under the action of the Gelfand-Tsetlin torus.
\end{lemma}
To make the notation easier, we will write $A \sim B$ if $A$ can be conjugated to $B$ using a special orthogonal matrix of appropriate size. We also write $(A)_l$ for the $l \times l $ top left submatrix of $A$.
\begin{proof} Applying the Propositions \ref{openeven} and \ref{openodd} we deduce that, in the $G=SO(2n+1)$ case, 
 $\Lambda \inv (\eo{E})$ consists of matrices $M$,\footnotesize
\begin{displaymath}\left(
\begin{array}{c|c}
\begin{array}{c|c}
 \begin{array}{c|c} \begin{array}{ccccccc}
L(\lambda_{1})&&&&&&\\
&\ddots&&&&&\\
&&L(\lambda_{j-1})&&&&\\
&&&L(x^{(2k)}_{j})&&&\\
&&&&L(\lambda_{j+1})&&\\
&&&&&\ddots&\\
&&&&&&L(\lambda_k)
\end{array} &P\\ \hline -P^T &0 \end{array} &Y\\ \hline -Y^T &0
\end{array}
 &0\\
\hline
0&\begin{array}{cccc}
L(\lambda_{k+2})&&&\\
&\ddots&&\\
&&L(\lambda_{n})&\\
&&&0
\end{array} \end{array} \right)
\end{displaymath} 
\normalsize
where $$x^{(2k)}_{j}\in ( \lambda_{j+1},\lambda_j ]\textrm{ if }j < k,$$ $$ x^{(2k)}_{j}\in (-\lambda_{k},\lambda_k ]\textrm{ if }j=k,$$ and the real vectors $P$ and $Y$ are such that $$(M)_{2k+1} \sim (\lambda)_{2k+1}\textrm{ and }(M)_{2k+2} \sim (\lambda)_{2k+2}.$$
Top right $(2k+2)\times (2n+1-2k-2)$ minor, and bottom left $(2n+1-2k-2) \times (2k+2)$ minor of $M$ must be zero in order to have $(M)_{l}\sim (\lambda)_{l}$ for all $l >2k+2$.

The Proposition \ref{openodd} implies that the $l$-th coordinate of $P$,$p_l$, must be zero for all $l \neq 2j-1, 2j$. The traces of $(\,(M)_{2k+1}\,)^2$ and $(\,(\lambda)_{2k+1}\,)^2$
need to be equal, therefore $p_{2j-1}^2+p_{2j}^2=\lambda_j^2-(x^{(2k)}_{j})^2$. 
This gives a circle of solutions for every choice of $x^{(2k)}_{j}$ in $(\lambda_{j+1},\lambda_{j})$ and the unique solution of $p_{2j-1}^2=p_{2j}^2=0$ if $x^{(2k)}_{j}=\lambda_j$.

Now we analyze conditions on vector $Y$. We are to have that $(M)_{2k+2}\sim (\lambda)_{2k+2}.$ If $B\in SO(2k+1)$ is such that $B\,(M)_{2k+1}\,B\inv=(\lambda)_{2k+1}$, then
\begin{displaymath}
 \left(\begin{array}{c|c}B&\\ \hline &1\end{array}\right)\,
\left(\begin{array}{c|c}(M)_{2k+1}&Y\\ \hline-Y^T &0\end{array}\right)\,
\left(\begin{array}{c|c}B\inv&\\ \hline &1\end{array}\right)\,=
\left(\begin{array}{c|c}(\lambda)_{2k+1}&BY\\ \hline -Y^TB\inv&0\end{array}\right).
\end{displaymath}
Therefore \begin{displaymath}
       \left(\begin{array}{c|c}(\lambda)_{2k+1}&BY\\ \hline -Y^TB\inv&0\end{array}\right) \sim \left(\begin{array}{c|c}(M)_{2k+1}&Y\\ \hline-Y^T &0\end{array}\right)=(M)_{2k+2} \sim (\lambda)_{2k+2}.    
          \end{displaymath}
Denote the coordinates of the vector $BY$ by $(v_1, \ldots, v_{2k+1})$. According to the Lemma \ref{openeven} the condition that 
\begin{displaymath}
 \left(\begin{array}{c|c}(\lambda)_{2k+1}&BY\\ \hline -Y^TB\inv&0\end{array}\right) \sim (\lambda)_{2k+2}
\end{displaymath}
implies that
$$v_1=\ldots=v_{2k}=0,\,\,\,v_{2k+1}^2=\lambda_{k+1}^2.$$
Therefore
\begin{equation}\label{ysolution}
 BY=\left(\begin{array}{c}0\\ \vdots \\0\\ \lambda_{k+1}\end{array}\right)\,\textrm{ or }BY=\left(\begin{array}{c}0\\ \vdots \\0\\ -\lambda_{k+1}\end{array}\right).
\end{equation}
For any choice of vector $P$, matrix $B$ is uniquely defined only up to multiplication by an element of maximal torus of $SO(2k+1)$. Every element $t$ of this torus has $k$ 
$(2 \times 2)$ blocks of rotations on the diagonal, the last diagonal entry equal to $1$, and all other entries zero. Therefore we have exactly two 
solutions to equation (\ref{ysolution}): $$Y=  B\inv (0,\ldots,0,\pm \lambda_{k+1})^T.$$
For both of these solutions $(M)_{2k+2}$
has the desired characteristic polynomial 
$q_{2k+2}(t)=\prod_{l=1}^{k+1}(t^2+ \lambda_l^2)$. However only one of them will give us matrix in the $SO(2k+2)$-orbit of $(\lambda)_{2k+2}$ as explained
in the proof of Proposition \ref{ineqeven}.
This means that the vector $Y$ is uniquely defined for every choice of vector $P$. Therefore the preimage of $E$ is a disk.

If $G=SO(2n)$ the proof is nearly identical. Just delete last row and column in the presentation of $M$. Conditions on $X$ and $Y$ stay the same.

\end{proof}
Now we analyze the weights of the action.
\begin{lemma}\label{weighte}
 The weight of the Gelfand-Tsetlin torus on $T_{\lambda}\Lambda \inv (\eo{E})$ is $-w^{(2k)}_{j}$, where
$$w^{(2k)}_{j}:=\sum_{l=2j}^{2k}x^{(l)}_{j}$$
and $E$ is an edge of $\cc{P}$ equal to the vector
\begin{align*}
 \,\langle \, (e_j-e_{j+1})^{\vee}, \lambda \,\rangle\,\,w^{(2k)}_{j}&=
(\lambda_j-\lambda_{j+1})\,w^{(2k)}_{j}&\textrm{ if }j<k,\\
\langle e_k^{\vee},\lambda \rangle\,w^{(2k)}_{k}&=
2\,\lambda_k\,w^{(2k)}_{k}&\textrm{ if }j=k \end{align*}
\end{lemma}
\begin{remark}
 Lemmas \ref{weighte} and \ref{weightf} find all the isotropy weights of the Gelfand-Tsetlin torus action at $\lambda$. Consider the lattice generated by the weights. Notice that for the special orthogonal group the weights are primitive vectors in the lattice they generate. This fact has an important consequence. To apply Proposition \ref{embed} we need to find $c$ such that the set $E^{(2k)}_{j}$ is equal to the $(-c)$ times the isotropy weight along $E$. 
In our case, the $c$ we need is the same as the lattice length of $E$ with respect to the weight lattice, exactly because all the weights are primitive. We want to point out that this is not necessarily true in general.
\end{remark}

\begin{proof} To make notation easier we concentrate on the case $G=SO(2n+1)$. The proof for $G=SO(2n)$ is nearly identical.

An element $R \in T_{SO(l)}$ of maximal torus of $SO(l)$, with $l \geq 2k+2$, acts on a matrix $M \in \Lambda \inv (\eo{E})$ by conjugation with
\begin{displaymath}
 \left(\begin{array}{c|c}B\inv R B &\\ \hline &I_{2n+1-l}\end{array}\right)\end{displaymath}
where $B \in SO(l)$ is such that $B(M)_{l}B\inv =(\lambda)_l \in (\lie{t}_{SO(l)})^*_+$. This action is trivial. 
To see this denote by $S$ the bottom left $(n+1-l) \times (n+1-l)$ submatrix of $M$. Then 
\begin{displaymath}
 \left(\begin{array}{c|c}B\inv R B &\\ \hline & I \end{array} \right)\,
\left(\begin{array}{c|c}(M)_{l}&0\\ \hline 0 & S\end{array}\right)\,
\left(\begin{array}{c|c}B\inv R\inv B&\\ \hline &I\end{array}\right)\,=\,
\left(\begin{array}{c|c}(M)_{l}&0\\ \hline 0 & S\end{array}\right).
\end{displaymath}
Therefore the functions $x^{(l)}_*$ with $l \geq 2k+2$ are constant on $\Lambda \inv (\eo{E})$.

Now consider the action of maximal torus of $SO(2k+1)$, $T_{SO(2k+1)}$. Let $B \in SO(2k+1)$ be such that 
$B(M)_{2k+1}B\inv =(\lambda)_{2k+1}\in (\lie{t}_{SO(2k+1)})^*_+$. Denote by $S$ the bottom right $(2n-2k) \times (2n-2k)$ submatrix of $M$. An element $R$ of $T_{SO(2k+1)}$ has the form \small
\begin{displaymath}
 R=\left( \begin{array}{cccc}
R(\alpha_1)&&&\\ & \ddots &&\\ && R(\alpha_k)& \\ &&& 1 \end{array}\right)
\end{displaymath} \normalsize
and it acts on $M$ by
\small
\begin{displaymath}
 \left(\begin{array}{c|c}B\inv R B &\\ \hline & I_{2n-2k} \end{array} \right)\,
\left(\begin{array}{c|c}(M)_{2k+1}&\left( \begin{array}{c|c}Y&0
\end{array}\right)\\ \hline   & \\ \left( \begin{array}{c}-Y^T\\ \hline 0
\end{array}\right)& S\end{array}\right)\,
\left(\begin{array}{c|c}B\inv R\inv B&\\ \hline &I_{2n-2k}\end{array}\right)\,\end{displaymath}
\begin{displaymath}=
\left(\begin{array}{c|c}(M)_{2k+1}&\left( \begin{array}{c|c}B\inv R\inv B\,Y&0
\end{array}\right)\\ \hline &\\ \left( \begin{array}{c}-Y^T\,(BRB\inv)^T \\ \hline 0
\end{array}\right)& S\end{array}\right).
\end{displaymath}
\normalsize
Recall that 
\begin{displaymath}
 BY=\left( \begin{array}{c}0\\ \vdots \\ 0 \\ \pm \lambda_{k+1} \end{array}\right), \textrm{ so } 
RBY=\left( \begin{array}{c}0\\ \vdots \\ 0 \\ \pm \lambda_{k+1} \end{array}\right)=B\,Y,\,\textrm{ and } B\inv RBY=Y.
\end{displaymath}
Therefore this action is also trivial.

 Now let $T_{SO(l)}$ be the chosen maximal torus of $SO(l)$ for $l \leq 2k$. An element of $T_{SO(l)}$ is of the form  $R=\diag(R(\alpha_1),\ldots, R(\alpha_{\frac{l-1}{2}}),1)$ or $R=\diag(R(\alpha_1),\ldots, R(\alpha_{\frac l 2}))$. Note that for $l \leq 2k$ the submatrix $(M)_{l}$ is in the positive Weyl chamber $(\lie{t}_{SO(l)})^*_+$. Therefore an element  $R \in T_{SO(l)}$ acts on $M$ simply by conjugation. Denote by $W$ the top right $l \times (2n+1-l)$ submatrix of $M$, and by $S$ the bottom right $(2n+1-l) \times (2n+1-l)$ submatrix of $M$. With this notation, the action of $R$ is the following.
\begin{displaymath}
 \left(\begin{array}{c|c} R &\\ \hline & I \end{array} \right)\,
\left(\begin{array}{c|c}(M)_{l}& W\\ \hline  & \\ -W^T & S\end{array}\right)\,
\left(\begin{array}{c|c}R\inv &\\ \hline &I\end{array}\right)\,=\,
\left(\begin{array}{c|c}(M)_{l}&RW\\ \hline & \\  -(RW)^T & S\end{array}\right).
\end{displaymath}
Only two of the columns of $W$ maybe be non-zero:
column $(2k+2)$-nd contains the first $l$ coordinates of the vector $Y$, and column $(2k+1)$-st contains the first $l$ coordinates of the vector $P$. 
We already showed that the only possibly non-zero entries of the vector $P$ are $p_{2j-1}$ and $p_{2j}$. 
Therefore the submatrix $W$ has possibly non-zero entries in the $(2k+1)$-st column if and only if $l \geq 2j$. 
In this case, notice that only the $j$-th circle of $T_{SO(l)}$ acts on the $(2k+1)$-st column, with speed $1$.
\begin{displaymath}
 R \left( \begin{array}{c} 0\\ \vdots \\ 0 \\ p_{2j-1}\\p_{2j}\\0\\ \vdots \\ 0 \end{array} \right)=
 \left( \begin{array}{c} 0\\ \vdots \\ 0 \\ R(\alpha_j) \left( \begin{array}{c} p_{2j-1}\\p_{2j}\end{array}\right) \\0\\ \vdots \\ 0 \end{array} \right).
\end{displaymath}
Recall that the vector $Y$ is uniquely determined by the vector $P$. Therefore, when we analyze the action of $T$ on $T_{\lambda}\Lambda \inv (\eo{E})$, independent variables are only in $W,S, P$. This means that the weight of the Gelfand-Tsetlin torus on $T_{\lambda}\Lambda \inv (\eo{E})$ is 
$$-w^{(2k)}_{j}:=-\sum_{l=2j}^{2k}\,x^{(l)}_{j}.$$
The conditions (\ref{edgeeven}) imply that the set $E$ is an edge of the polytope $\cc{P}$ given by the vector 
$$(\lambda_{j}-\lambda_{j+1})\,w^{(2k)}_{j}=
 \,\langle \, (e_j-e_{j+1})^{\vee},\lambda \,\rangle \, \,w^{(2k)}_{j},$$
if $j<k$, and by the vector 
$$\,\langle \, e_k^{\vee},\lambda \,\rangle\,\,w^{(2k)}_{k}=
2\,\lambda_k\,w^{(2k)}_{k}$$ if $j=k.$

Recall that for $G=SO(2n+1)$ we were taking $k$ from the set $\{1, \ldots, n\}$, and for $G=SO(2n)$ we had $k\in\{1, \ldots, n-1\}$.
Therefore the collection of lattice lengths of edges $E^{(2k)}_{j}$ is $$\{\lambda_1-\lambda_2, \ldots, \lambda_{n-1}-\lambda_n, 2\lambda_1, \ldots, 2\lambda_n\}\textrm{ for }G=SO(2n+1)$$ $$\{\lambda_1-\lambda_2, \ldots, \lambda_{n-2}-\lambda_{n-1}, 2\lambda_1, \ldots, 2\lambda_{n-1}\}\textrm{ for }G=SO(2n).$$
 \end{proof}


Now we analyze the other edges starting from $\Lambda(\lambda)$.		
We still think of $\bb{R}^{N}$ as having coordinates $\{x^{(k)}_{j}\}$, for appropriate $k,j$. 
Pick any $k < n$ and $j \leq k$. Consider the set $F:=F^{(2k+1)}_{j}$, that is the image of points where all the Gelfand-Tsetlin functions are equal to their upper bound, apart from the 
function $\lambda^{(2k+1)}_{j}$. That is, $F$ is the set of points satisfying
 \begin{align}\label{edgeodd}
x^{(m)}_{l}&=\lambda_l \textrm{ for all }m\textrm{ and all }l\neq j,\notag \\
x^{(m)}_{j}&=\lambda_j\textrm{ for all }m \geq 2 k+2,\\ 
x^{(m)}_{j}&=x^{(2k+1)}_{j}\textrm{ for all }2j\leq m \leq 2k+1,\notag
\end{align}
where $x^{(2k+1)}_{j}\in [\lambda_{j+1}, \lambda_j]$, unless $G=SO(2n)$ and $k=n-1, j=n-1$ when $x^{(2n-1)}_{n-1}\in [|\lambda_{n}|, \lambda_{n-1}].$
 Here is graphical presentation
\scriptsize
\begin{equation}
  \begin{array}{ccccccccccccc}
 \lambda_{j-1}&&&\lambda_{j}&&&\lambda_{j+1}&&&&&&\\
&\begin{rotate}{135} = \end{rotate} &&&\begin{rotate}{135} = \end{rotate}&&&\begin{rotate}{135} = \end{rotate}&&&&\\
&&x^{(2k+2)}_{j-1} &&&x^{(2k+2)}_{j}&&&x^{(2k+2)}_{j+1}&&&\\
&&& \begin{rotate}{135} = \end{rotate} &&&\begin{rotate}{-45} $>$ \end{rotate}&&&\begin{rotate}{135} = \end{rotate}&&\\
&&&& x^{(2k+1)}_{j-1}&&&x^{(2k+1)}_{j}&&&x^{(2k+1)}_{j+1}&\\
&&&&&\begin{rotate}{135} = \end{rotate} &&&\begin{rotate}{135} = \end{rotate}&&&\begin{rotate}{135} = \end{rotate}\\
&&&&&&x^{(2k)}_{j-1}&&&x^{(2k)}_{j}&&&x^{(2k)}_{j+1}\\
 \end{array}
\end{equation}
\normalsize
Again, similarly to the unitary case (\cite[Lemma 3.10]{P}), one can show that $F$ is an edge of $\cc{P}$. Let $\eo{F}=\Lambda(\lambda) \cup \textrm{int }F$ denote the edge $F$ without the second endpoint.
From the definition of $U$ and the fact that $\Lambda(\lambda)\in U$, it follows that the set $\Lambda \inv (\eo{F})$ is also contained in $U$. Therefore it is equipped with a smooth action of the Gelfand-Tsetlin torus.
\begin{lemma}\label{sphereb}
 $\Lambda \inv (\eo{F} )$ is a disc invariant under the action of the Gelfand-Tsetlin torus.
\end{lemma}
\begin{proof} In this proof we again concentrate on the case $G=SO(2n+1)$ as the procedure for $G=SO(2n)$ is analogous.

 $\Lambda \inv (\eo{F})$ consists of matrices $M$ of the form\scriptsize
\begin{displaymath}\left(
\begin{array}{c|c}
 \begin{array}{c|c} \begin{array}{cccccccc}
L(\lambda_{1})&&&&&&&0\\
&\ddots&&&&&&\\
&&L(\lambda_{j-1})&&&&&\\
&&&L(x^{(2k+1)}_{j})&&&&\vdots\\
&&&&L(\lambda_{j+1})&&&\\
&&&&&\ddots&&\\
&&&&&&L(\lambda_k)&0\\
0&&&\ldots&&&0&0
\end{array} &Y\\ \hline -Y^T &0 \end{array} 
 &0\\
\hline
0&\begin{array}{cccc}
L(\lambda_{k+2})&&&\\
&\ddots&&\\
&&L(\lambda_{n})&\\
&&&0
\end{array} \end{array} \right)
\end{displaymath} 
\normalsize
where $x^{(2k+1)}_{j}\in (\lambda_{j+1}, \lambda_j]$, unless $G=SO(2n)$ and $k=j=n-1$ when $x^{(2n-1)}_{n-1}\in (|\lambda_{n}|, \lambda_{n-1}],$ and the real vector $Y$ is such that 
$(M)_{2k+2} \sim (\lambda_{2k+2})$.
 Notice that the top right and bottom left minors have to be zero to have that $(M)_{l}\sim (\lambda)_{l}$ for any $l >2k+2$. The Proposition \ref{openeven} implies that 
$$y_l=0\textrm{ for all }l \neq 2j-1,\, 2j,\, 2k+1$$ and that $y_{2k+1}$ and $y_{2j-1}^2+y_{2j}^2$ are uniquely defined.
If $x^{(2k+1)}_{j}=\lambda_{j}$, then $y_{2j-1}=y_{2j}=0$, and $y_{2k+1}=-\lambda_{j}.$ 
For each $x^{(2k+1)}_{j}\in (\lambda_{j+1},\lambda_{j})$ we have a circle worth of choices for $y_{2j-1},y_{2j}=0$, and unique choice for $y_{2k+1}$. 
Therefore $\Lambda \inv (\eo{F}) $ is a $2$-dimensional disk.
\end{proof}
Now we analyze the weights of the action.
\begin{lemma}\label{weightf}
 The weight of the Gelfand-Tsetlin torus on $T_{\lambda}\Lambda \inv (\eo{F})$ is $-w^{(2k+1)}_{j}$, where
$$w^{(2k+1)}_{j}:=\sum_{l=2j}^{2k+1}\,x^{(l)}_{j}$$
and $F$ is an edge of $\cc{P}$ equal to the vector
$$  \,\langle\, (e_j-e_{j+1})^{\vee},\lambda \,\rangle\,\,w^{(2k+1)}_{j}=(\lambda_j-\lambda_{j+1})\,w^{(2k+1)}_{j},$$
unless $G=SO(2n)$ and $k=n-1, j=n-1$ when $F$ is a subset of an edge of $\cc{P}$ equal to the vector $(\lambda_{n-1}-|\lambda_{n}|)\,w^{(2n-1)}_{n-1}.$
\end{lemma}
\begin{proof} For simplicity of notation assume that $G=SO(2n+1)$. To obtain the proof in the case $G=SO(2n)$ one only needs to delete the last row and column of $M$.

First consider the action of $T_{SO(l)}$ with $l \geq 2k+2$. An element $R \in T_{SO(l)}$ of the maximal torus of $SO(l)$ acts on matrix $M \in \Lambda \inv (\eo{F})$ by conjugation with
\begin{displaymath}
 \left(\begin{array}{c|c}B\inv R B &\\ \hline &I_{2n+1-l}\end{array}\right)\end{displaymath}
where $B \in SO(l)$ is such that $B(M)_{l}B\inv =(\lambda)_l \in (\lie{t}_{SO(l)})^*_+$. 
Denote by $S$ the bottom left $(n+1-l) \times (n+1-l)$ submatrix of $M$. Have
\begin{displaymath}
 \left(\begin{array}{c|c}B\inv R B &\\ \hline & I \end{array} \right)\,
\left(\begin{array}{c|c}(M)_{l}&0\\ \hline 0 & S\end{array}\right)\,
\left(\begin{array}{c|c}B\inv R\inv B&\\ \hline &I\end{array}\right)\,=\,
\left(\begin{array}{c|c}(M)_{l}&0\\ \hline 0 & S\end{array}\right).
\end{displaymath}
Therefore the functions $x^{(l)}_*$ for $l \geq 2k+2$ are constant on $\Lambda \inv (\eo{F})$ and the action is trivial. 

Now consider the action of $T_{SO(l)}$, for $l \leq 2k+1$. 
An element $R$ of $T_{SO(l)}$ has the form 
\begin{displaymath}
 R=\left( \begin{array}{cccc}
R(\alpha_1)&&&\\ & \ddots &&\\ && R(\alpha_{\lfloor \frac l 2 \rfloor})& \\ &&& 1 \end{array}\right)\textrm{  or  }R=\left( \begin{array}{ccc}
R(\alpha_1)&&\\ & \ddots &\\ && R(\alpha_{ \frac l 2 }) \end{array}\right).\end{displaymath}
Denote by $W$ the top right $l \times (2n+1-l)$ submatrix of $M$, and by $S$ the bottom right $(2n+1-l) \times (2n+1-l)$ submatrix of $M$. 
Notice that $(M)_l \in (\lie{t}_{SO(l)})^*_+$. Therefore the action of $R$ is the following.
\begin{displaymath}
 \left(\begin{array}{c|c} R &\\ \hline & I \end{array} \right)\,
\left(\begin{array}{c|c}(M)_{l}& W\\ \hline -W^T & S\end{array}\right)\,
\left(\begin{array}{c|c}R\inv &\\ \hline &I\end{array}\right)\,=\,
\left(\begin{array}{c|c}(M)_{l}&RW\\ \hline -(RW)^T & S\end{array}\right).
\end{displaymath}
Only one of the columns of $W$ maybe be non-zero:
column $(2k+2)$-nd contains the first $l$ coordinates of the vector $Y$. 
We already showed that the only possibly non-zero entries of the vector $Y$ are $y_{2j-1}$, $y_{2j}$  and $y_{2k+1}$.
Therefore the submatrix $W$ has possibly non-zero entries in the $(2k+1)$-st column if and only if $l \geq 2j-1$. 
The action does not change the $(2k+1, 2k+1)$-th entry of $M$, namely $y_{2k+1}$. This is because this entry is a part of $W$ only in the case $l=2k+1$. In that case, $R$ acts on this entry by multiplication by its $(2k+1, 2k+1)$-th entry, which is equal to $1$. 
There is however nontrivial action on the $(2k+1, 2j-1)$-th and $(2k+1, 2j)$-th entries of $M$ if only $l \geq 2j$.
The $j$-th circle of $T_{SO(l)}$ acts on the $(2k+1)$-st column, rotating them with speed $1$.
\begin{displaymath}
 R \left( \begin{array}{c} 0\\ \vdots \\ 0 \\ y_{2j-1}\\y_{2j}\\0\\ \vdots \\ 0 \end{array} \right)=
 \left( \begin{array}{c} 0\\ \vdots \\ 0 \\ R(\alpha_j) \left( \begin{array}{c} y_{2j-1}\\y_{2j}\end{array}\right) \\0\\ \vdots \\ 0 \end{array} \right).
\end{displaymath}
This means that the weight of the Gelfand-Tsetlin torus on $T_{\lambda}\Lambda \inv (\eo{F})$ is 
$$-w^{(2k+1)}_{j}:=-\sum_{l=2j}^{2k+1}\,x^{(l)}_{j}.$$
The condition (\ref{edgeodd}) implies that $F$ is a subset of an edge of $\cc{P}$ equal to the vector
$$(\lambda_j-\lambda_{j+1})w^{(2k+1)}_{j},$$
unless $G=SO(2n)$ and $k=n-1, j=n-1$ when $F$ is equal to the vector $(\lambda_{n-1}-~|\lambda_n|)w^{(2n-1)}_{n-1}.$ 

Note the collection of lattice lengths of edges $F^{(2k+1)}_{j}$ is 
$$\{\lambda_1-\lambda_2, \ldots, \lambda_{n-1}-\lambda_n\}\textrm{ for }G=SO(2n+1),$$ 
$$\{\lambda_1-\lambda_2, \ldots, \lambda_{n-2}-~\lambda_{n-1},\,  \lambda_{n-1}-~|\lambda_n|\}\textrm{ for }G=SO(2n).$$ 
\end{proof}
We summarize the above section in the following corollary.
\begin{corollary}\label{minedgelengths}
 Every edge of $\cc{P}$ starting from $\Lambda(\lambda)$ has lattice length equal to at least $\min\{\, \left|\left\langle \alpha^{\vee},\lambda \right\rangle \right|\,; \alpha^{\vee} \textrm{ a  coroot}\}.$
\end{corollary}
\begin{proof}
Direct application of Lemmas \ref{weighte} and \ref{weightf} would give us lower bounds for lattice lengths equal to 
$$\min\{\lambda_1-\lambda_2, \ldots, \lambda_{n-1}-\lambda_n, 2\,\lambda_1, \ldots ,2\,\lambda_n\}\textrm{ if }G=SO(2n+1),$$
$$\min\{\lambda_1-\lambda_2, \ldots, \lambda_{n-1}-\lambda_n, \lambda_{n-1}-|\lambda_n|,2\,\lambda_1, \ldots ,2\,\lambda_{n-2},2\,\lambda_{n-1}\}\textrm{ if }G=SO(2n).$$
Inequalities coming from the fact that $\lambda$ is in the positive Weyl chamber imply that the minimum over the first set is equal to 
$$\min \{\lambda_1-\lambda_2, \ldots, \lambda_{n-1}-\lambda_n, 2\,\lambda_n\},$$
while the minimum over the second set is equal to
$$\min \{\lambda_1-\lambda_2, \ldots, \lambda_{n-1}-\lambda_n, \lambda_{n-1}+\lambda_n\}.$$
For example, $$2\,\lambda_{n-1}>\lambda_{n-1}+\,|\lambda_{n}|=\lambda_{n-1}\pm \lambda_n,$$ $$\lambda_{n-1}-|\lambda_n|=\min\{\lambda_{n-1}-\lambda_n, \lambda_{n-1}+\lambda_n\}.$$
Analysis of root systems done in Subsection \ref{rootsystem} gives that in both cases the minimum is equal to $\min\{\, \left|\left\langle \alpha^{\vee},\lambda \right\rangle \right|\,; \alpha^{\vee} \textrm{ a  coroot}\}.$
\end{proof}
\section{The proof of the Main Theorem}.\label{proof} \begin{proof}
To prove the Theorem \ref{main}, we will proceed as in the Example \ref{centeredexample}. Recall that $2N$ is the dimension of the orbit $\cc{O}_{\lambda},$ where $N=n^2$ if $G=SO(2n+1)$ and $N=n(n-1)$ if $G=SO(2n)$. The point $\lambda \in \cc{O}_{\lambda}$ is a fixed point for the action of the Gelfand-Tsetlin torus. Moreover, preimage of $\Lambda (\lambda)$ is a single fixed point, $\{\lambda\}$. From the definition of $U$ it follows that $\lambda \in U$ and that
$$\cc{T}:= \bigcup\limits_{\substack{\cc{F} \text{ face of } \cc{P} \\ \Lambda (\lambda) \in \cc{F}}}
   \Lambda^{-1}(\text{rel-int } \cc{F}) \,\subset U.$$ Moreover the action of the Gelfand-Tsetlin torus on $\cc{T}$ is centered around $\Lambda (\lambda)$. Denote the weights of the action $T_{GT} \curvearrowright T_{\lambda}\cc{T}=T_{\lambda}\cc{O}_{\lambda}$ by $-\eta_1, \ldots, -\eta_N.$   Let
$r=\min\{\, \left|\left\langle \alpha^{\vee},\lambda \right\rangle \right|\,; \alpha^{\vee} \textrm{ a  coroot}\}.$
Corollary \ref{minedgelengths} shows that lattice lengths of all edges starting from $\Lambda (\lambda)$ are at least $r$.
Therefore 
$$\Lambda(\lambda) + \pi \sum_{i=1}^N |z_i|^2 \eta_i \in \cc{T}$$ for any $z \in B^{2N}_r,$ 
ball of capacity $r$.
Proposition \ref{embed} gives symplectic embedding of the ball of the capacity $r$. 
Therefore $r$ is the lower bounds for Gromov width. \end{proof}
\section{Orbits that are not regular.}\label{notregular}
In this section we analyze orbits that intersect the positive Weyl chamber at a point on the boundary of the chamber. Therefore they are not regular. In the literature they are often referred to as non-generic orbits.
In the case of the unitary group, the Gelfand-Tsetlin action allows to calculate the lower bound for Gromov width also for some class of such orbits, \cite{P}.
For the $SO(2n+1)$ the Theorem \ref{main} can also be generalized to a class of orbits that are not regular. The same argument applied in the case $G=SO(2n)$ gives only a lower bound that is smaller then the expected one. We still present it here as no lower bounds were previously known.
\begin{theorem}\label{nongeneric}
 Let $\lambda$ be a block diagonal matrix
\begin{displaymath}
\lambda= \begin{cases}
  \diag\,(\,L(\lambda_1), \ldots, L(\lambda_n),1) \in \lie{t}^*_+=(\lie{t}_{SO(2n+1)})^*_+ &\textrm{ if }G=SO(2n+1)\\
  \diag\,(\,L(\lambda_1), \ldots, L(\lambda_n)) \in \lie{t}^*_+=(\lie{t}_{SO(2n)})^*_+ &\textrm{ if }G=SO(2n).\\
 \end{cases}
\end{displaymath}
Assume that 
$$\lambda_1 >...>\lambda_s=\lambda_{s+1}=\ldots =\lambda_{s+l-1} >\lambda_{s+l}>\ldots >\lambda_{n-1}>|\lambda_n|.$$
Then the Gromov width of the $G$ orbit $\cc{O}_{\lambda}$ through $\lambda$ is at least
$$\min\{\langle \alpha^{\vee},\lambda \rangle;\,\alpha^{\vee}\textrm{ a coroot and }\langle \alpha^{\vee},\lambda \rangle >0\}\textrm{ if }G=SO(2n+1)$$
$$\min\{ 2|\lambda_n|, \min\{\langle \alpha^{\vee},\lambda \rangle;\,\alpha^{\vee}\textrm{ a coroot and }\langle \alpha^{\vee},\lambda \rangle >0\}\}\textrm{ if }G=SO(2n).$$
\end{theorem}
\begin{proof}
 The dimension of the orbit is twice the number of Gelfand-Tsetlin functions that are not constant on the whole orbit. This is because the Gelfand-Tsetlin system is completely integrable for all orbits. To see that directly for the above orbit, calculate the dimension of the orbit from the $1$-skeleton of the momentum map image for the standard action of the maximal torus. This dimension is twice the number of edges in the $1$-skeleton starting at any vertex. Edges correspond to non-trivial permutations of $\lambda_j$'s. Therefore the dimension of the orbit is $2 {l \choose 2} =l(l-1)$ less then the dimension of a regular orbit. The number of Gelfand-Tsetlin functions that are forced to be constant on the whole orbit due to inequalities (\ref{polytopeineqodd}) and (\ref{polytopeineqeven}) is equal to $\frac{l(l-1)}{2}$. 
Propositions \ref{polytopeodd} and  \ref{polytopeeven} generalize to the case of not regular orbits as they were proved without any assumption on regularity.
Therefore in this case we again have that $\dim \cc{P} =\frac 1 2 \dim \cc{O}_{\lambda}$.
In this case, however, the point $\Lambda(\lambda)$ may not be in the set $U$ on which the Gelfand-Tsetlin are proved to be smooth and induce a smooth action. Consider the block diagonal matrix $\eta$
\begin{displaymath}
\eta= \begin{cases}
  \diag\,(\,L(\lambda_1), \ldots, L(\lambda_s),L(\lambda_{s+l}),\ldots, L(\lambda_n),L(\lambda_s),\ldots,L(\lambda_s),1)  &\textrm{ for }SO(2n+1)\\
  \diag\,(\,L(\lambda_1), \ldots, L(\lambda_s),L(\lambda_{s+l}),\ldots, L(\lambda_n),L(\lambda_s),\ldots,L(\lambda_s))&\textrm{ for }SO(2n).\\
 \end{cases}
\end{displaymath} That is, in the top left submatrix there are blocks $L(\lambda_j)$'s with $\lambda_j$'s all different, and the additional $L(\lambda_s)$ blocks are collected in the bottom right submatrix.
Let $V=\Lambda(\eta)$. Then $V$ is a vertex of the Gelfand-Tsetlin polytope $\cc{P}$ as each coordinate of $V$ is equal to its lower or upper bound (for more about identification of vertices of polytope see \cite{P} or \cite{Zi}). Figure \ref{vertexeta} presents equations satisfied by coordinates of $V$.
\begin{figure}
\includegraphics[width=.4\textwidth]{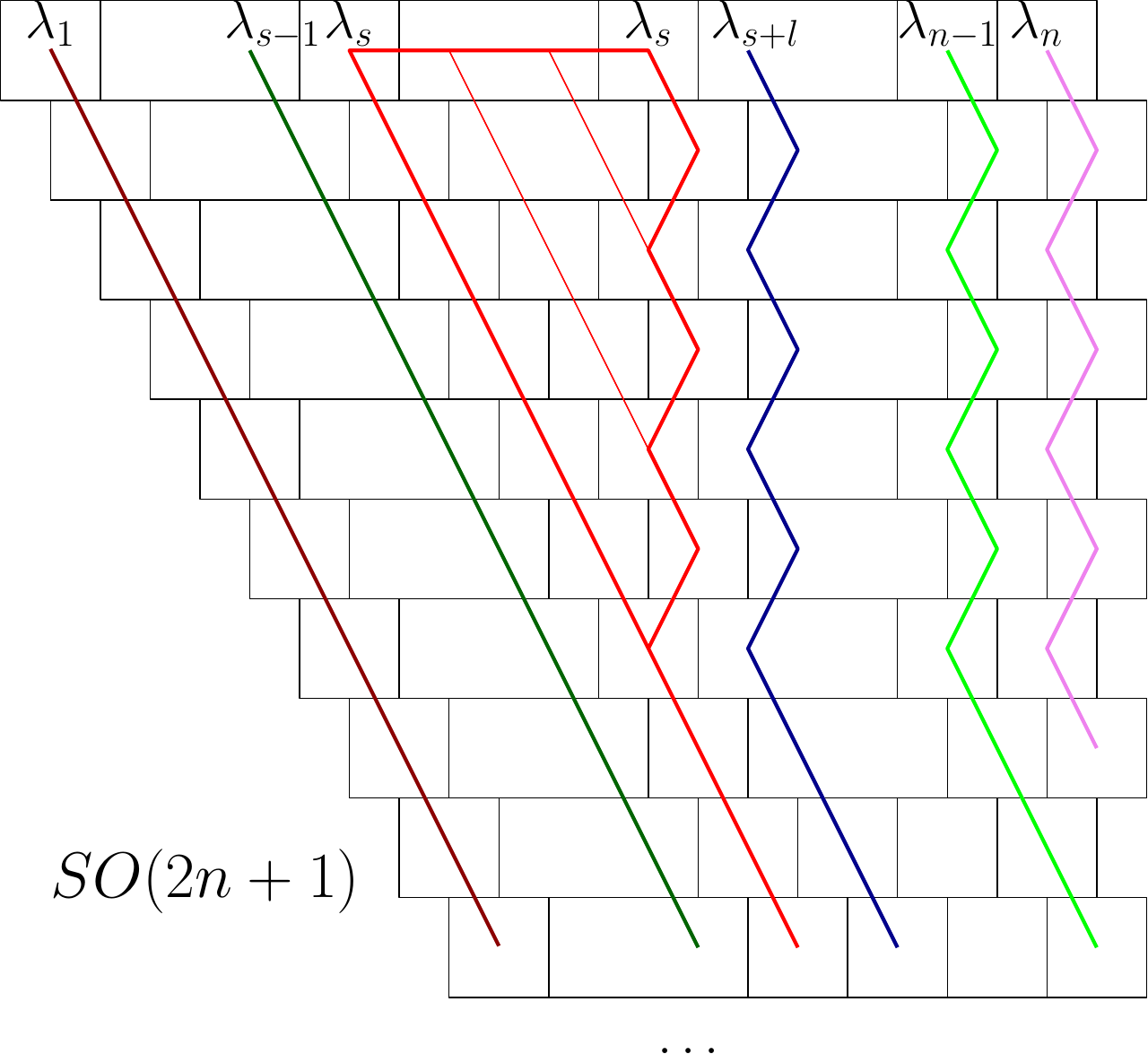} \hspace{20mm}
\includegraphics[width=.4\textwidth]{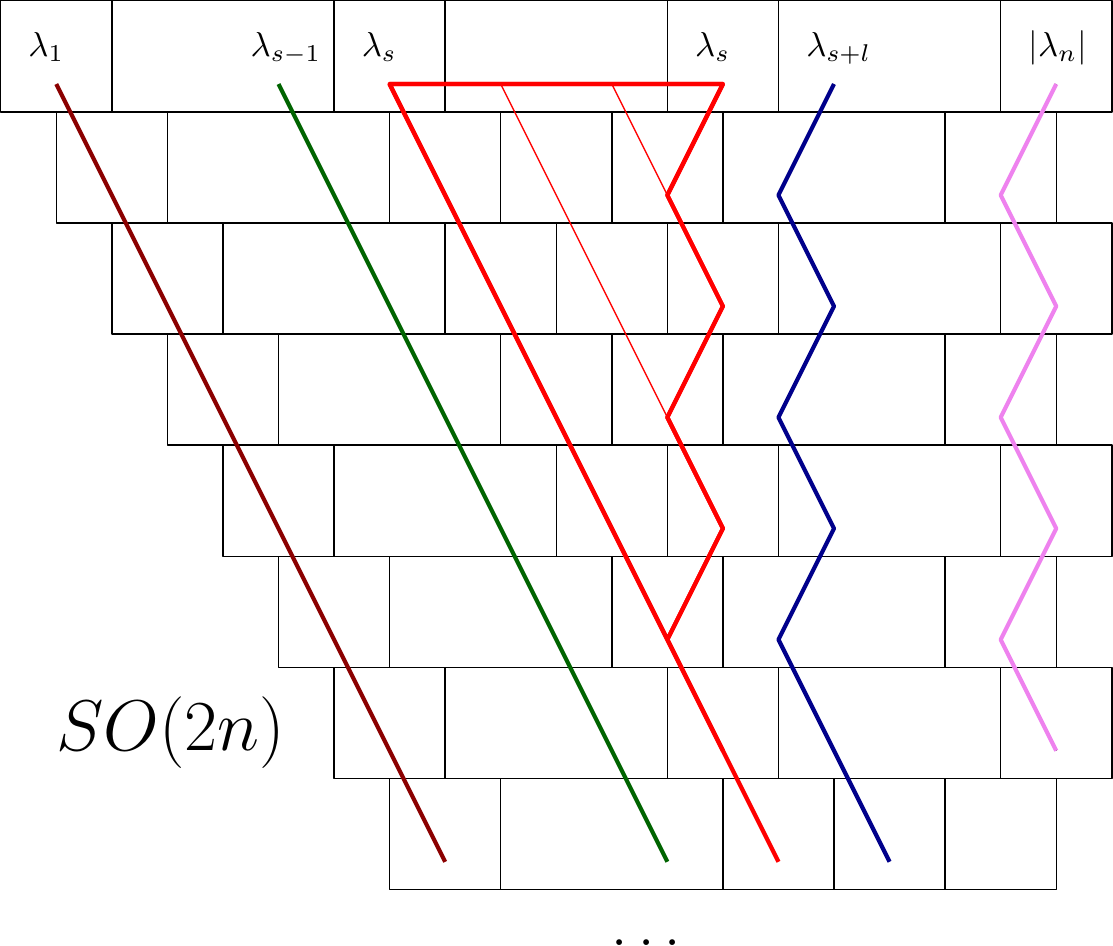}
\caption{Vertex $\Lambda(\eta)$ for the case $G=SO(2n+1)$ and $G=SO(2n)$.}\label{vertexeta}\end{figure}
The matrix $\eta$ is in $U$. 
Let $$\cc{T}= \bigcup\limits_{\substack{F \text{ face of } \cc{P} \\ V \in F}} 
   \Lambda \inv (\text{rel-int } F) .$$
Then from the definition of the set $U$ it follows that $\Lambda^{-1}(\cc{T}) \subset U$. Thus it is equipped with a smooth Gelfand-Tsetlin action and the subset $\cc{T}$ is centered around $V$. Similarly to the case of a regular orbit, we find edges of $\cc{P}$ starting from $V$ and their lengths with respect to the weights of the action. Notice that these lattice lengths are easy to read off from the triangle of equations satisfied by the vertex we start from. They are given by differences of values on two neighboring polylines in the triangle and by twice the value of the polyline hitting the right edge of the triangle. The same is true for not regular orbits, as the computations of lattice lengths is totally analogous. 
Therefore the lengths of the (subsets of) edges starting from $V$ in the $SO(2n+1)$ case are exactly
$$\{\lambda_1-\lambda_{2},\ldots,\lambda_{s-1}-\lambda_{s},\lambda_{s}-\lambda_{s+l},\ldots,\lambda_{n-1}-\lambda_{n},\,2\lambda_1, \ldots, 2\,\lambda_n\}.$$ 
The minimum over this set is equal to $\min\{\langle \alpha^{\vee},\lambda \rangle;\,\alpha^{\vee}\textrm{ a coroot and }\langle \alpha^{\vee},\lambda \rangle >0\}$ as claimed (compare with Corollary \ref{minedgelengths}.)
The lengths of the (subsets of) edges starting from $V$ in the $SO(2n)$ case are 
$$\{\lambda_1-\lambda_{2},\ldots,\lambda_{s-1}-\lambda_{s},\lambda_{s}-\lambda_{s+l},\ldots,\lambda_{n-1}-|\lambda_{n}|,\,2\lambda_1, \ldots,2\lambda_{n-1}, 2|\lambda_n|\}.$$
The minimum over this set is equal to 
$$\min\{ 2|\lambda_n|, \min\{\langle \alpha^{\vee},\lambda \rangle;\,\alpha^{\vee}\textrm{ a coroot and }\langle \alpha^{\vee},\lambda \rangle >0\}\} .$$
Similarly to the proof of the main theorem, we can apply the Proposition \ref{embed} and prove that the above values are lower bounds for Gromov width.
\end{proof}

\appendix
\section{Centered regions for non-simply laced groups.}\label{standardnotenough}
Let $G$ be a compact, connected, non-simply laced Lie group, and $T$ be a choice of maximal torus.
Choose positive Weyl chamber and let $p \in (\lie{t})^*_+$ be a point in the interior of this chamber. Consider the coadjoint orbit $M$, through $p$, and denote by $N$ the dimension of $M$. Coadjoint action of the maximal torus $T$ on $M$ is Hamiltonian. Denote the momentum map for this action by $\mu:M \rightarrow \lie{t}^*$. 
Let $\cc{Q}=\mu(\overline{\{x \in M;\,\,\dim (T\cdot x)=1\}})$ be the image of the $1$-skeleton of $M$. Then $\cc{Q}$ is an $N$-valent graph contained in the polytope $\mu (M)$. (This follows from the fact that $T$ acts on $M$ in a GKM fashion. For more about GKM manifolds see \cite{GKM}, \cite{TW}). Note that $p=\mu(p)$ is the fixed point of this action. Let $\cc{T} \subset \lie{t}^*$ be such that $\mu \inv (\cc{T})$ is centered around $p$. In particular, for any edge $E$ of $\cc{Q}$,  $E \cap \cc{T} \neq \emptyset$ if and only if $p \in E$. One could apply Proposition \ref{embed} and obtain some lower bound for Gromov width of $M$ as explained in the Example \ref{centeredexample}.
In this section we show that in the case of non-simply laced group, this lower bound is weaker (i.e. lower) then the predicted Gromov width of the  coadjoint orbit,
$$\min\left\{\left|\left\langle \alpha^{\vee},p \right\rangle\right|;\, \alpha^{\vee} \textrm{ a coroot }\right\}.$$
This observation makes our result for the $SO(2n+1)$ coadjoint orbits even more interesting, as the root system for $SO(2n+1)$ is non-simply laced. 

Let $\alpha,\beta \, \in \lie{t}^*$ be two roots of Euclidean lengths $||\alpha|| >||\beta||$.
For any root $\eta$ let $\sigma_{\eta}:\lie{t}^* \rightarrow \lie{t}^*$ denote the reflection through hyperplane perpendicular to $\eta$.
Then the image of $\alpha$ under the reflection $\sigma_{\beta}$,
$$\sigma_{\beta}(\alpha):= \alpha - 2 \frac{\langle \beta,\alpha\rangle}{\langle\beta,\beta\rangle}\beta=\alpha -\langle \beta^{\vee},\alpha\rangle\,\beta,$$ is also a root
(see condition R3 in III.9.2 of \cite{H}).
What is more, $$||\alpha||=|| \sigma_{\beta}(\alpha)||.$$
For any root $\eta$, the points $p$ and $\sigma_{\eta}(p)$ are connected by an edge of $\cc{Q}$. In particular there exist and edge in $\cc{Q}$ joining $p$ with a point 
$$\sigma_{\alpha}(p):=p-2 \frac{\langle\alpha,p\rangle}{\langle\alpha,\alpha\rangle}\alpha.$$ Call this edge $E_1$.
Denote by $E_2$ the edge in $\cc{Q}$ from $\sigma_{\beta}(p)$ in the direction of $\sigma_{\beta}(\alpha)$, 
joining $\sigma_{\beta}(p)$ with a vertex $\sigma_{\sigma_{\beta}(\alpha)}(\sigma_{\beta}(p))$ 
we will denote by $E_2$. The definition of centered region implies that the edge $E_2$ has to be disjoint from $\cc{T}$.
We want to know how big portion of the edge $E_1$ is contained in $\cc{T}$. Definitely the intersection of edges $E_1$ and $E_2$ is not in $\cc{T}$.
These edges intersect if there exists $t,s$ such that
$$ \sigma_{\beta}(p) + s \sigma_{\beta}(\alpha)=p+t\alpha.$$
This means:\begin{align*}
p+t\alpha&= \sigma_{\beta}(p) + s \sigma_{\beta}(\alpha)=
 p- 2 \frac{\langle\beta,p\rangle}{\langle\beta,\beta\rangle}\beta + s \left( \alpha - 2 \frac{\langle\beta,\alpha\rangle}{\langle\beta,\beta\rangle}\beta\right),\\
t\alpha&=- 2\, \frac{\langle\beta,p\rangle}{\langle\beta,\beta\rangle}\,\beta + s \alpha - 2 s\, \frac{\langle\beta,\alpha\rangle}{\langle\beta,\beta\rangle}\,\beta,\\
(t-s)\alpha&=- \, \frac{2}{\langle\beta,\beta\rangle}\,\left(\, \langle\beta , p\rangle+s \langle\beta,\alpha\rangle \,\right)\,\beta. \end{align*}
As $\alpha$ and $\beta$ are different of different lengths, the only solution to the above equation is when $t=s$ and $\langle\beta,p\rangle+s \langle\beta,\alpha\rangle =0$. The point $p$ was chosen from the interior of the positive Weyl chamber, thus $\langle\beta,p\rangle \neq 0$. 
The solution exists if also $\langle\beta,\alpha\rangle \neq0$ and is  
$$t=s=-\,\frac{\langle\beta,p\rangle}{\langle\beta,\alpha\rangle}=-2\,\frac{\langle\beta,p\rangle}{\langle\beta,\beta\rangle}\,\left(\frac{2\langle\beta,\alpha\rangle}{\langle\beta,\beta\rangle}\right)^{-1}.$$
The values of $\frac{2\langle\beta,\alpha\rangle}{\langle\beta,\beta\rangle}$ can only be $0,\pm 1,\pm 2,\pm 3$ (\cite[Chapter 9]{H}). By the above, we know it is not $0$.
If $\frac{2\langle\beta,\alpha\rangle}{\langle\beta,\beta\rangle}=\pm 1$, then $||\alpha||=||\beta||$ (\cite{H}) contrary to our assumptions.
Thus it has to be $\pm 2$ or $\pm 3$. In both cases we get that the solution 
$$|t|=2\left|\frac{\langle\beta,p\rangle}{\langle\beta,\beta\rangle}\,\left(\frac{2\langle\beta,\alpha\rangle}{\langle\beta,\beta\rangle}\right)^{-1}\right|<2\left|\frac{\langle\beta,p\rangle}{\langle\beta,\beta\rangle}\right|=\langle\beta^{\vee},p\rangle.$$
This means that the portion of the edge $E_1$ contained in $\cc{T}$ has length strictly less then $\langle\beta^{\vee},p\rangle \,||\alpha||.$
Therefore the lower bound for Gromov width that we can obtain from the centered region $\cc{T}$ is 
less then $\langle\beta^{\vee},p\rangle$ (the isotropy weight along the sphere $\mu \inv (E_1)$ is $(-\alpha)$).
It may happen that the minimum min$\{ \left|\,\langle\alpha_j^{\vee},p\rangle\,\right|; \alpha_j \textrm{ a  coroot}\}$ is equal to $\langle\beta^{\vee},p\rangle.$
In this case, the predicted lower bound of Gromov width of the orbit is strictly greater then the bound one could get from the centered region for the standard action of the maximal torus. 

For example, consider $SO(5)$ coadjoint orbit $M$ through a block diagonal matrix $p=\diag(L(6),L(1),1)$ in $\lie{so}(5)^*$.
The momentum polytope $\mu (M)$, together with the image of $1$-skeleton are presented on Figure \ref{notlacedeg}.
\begin{figure}\label{notlacedeg}
\includegraphics[width=0.6\textwidth]{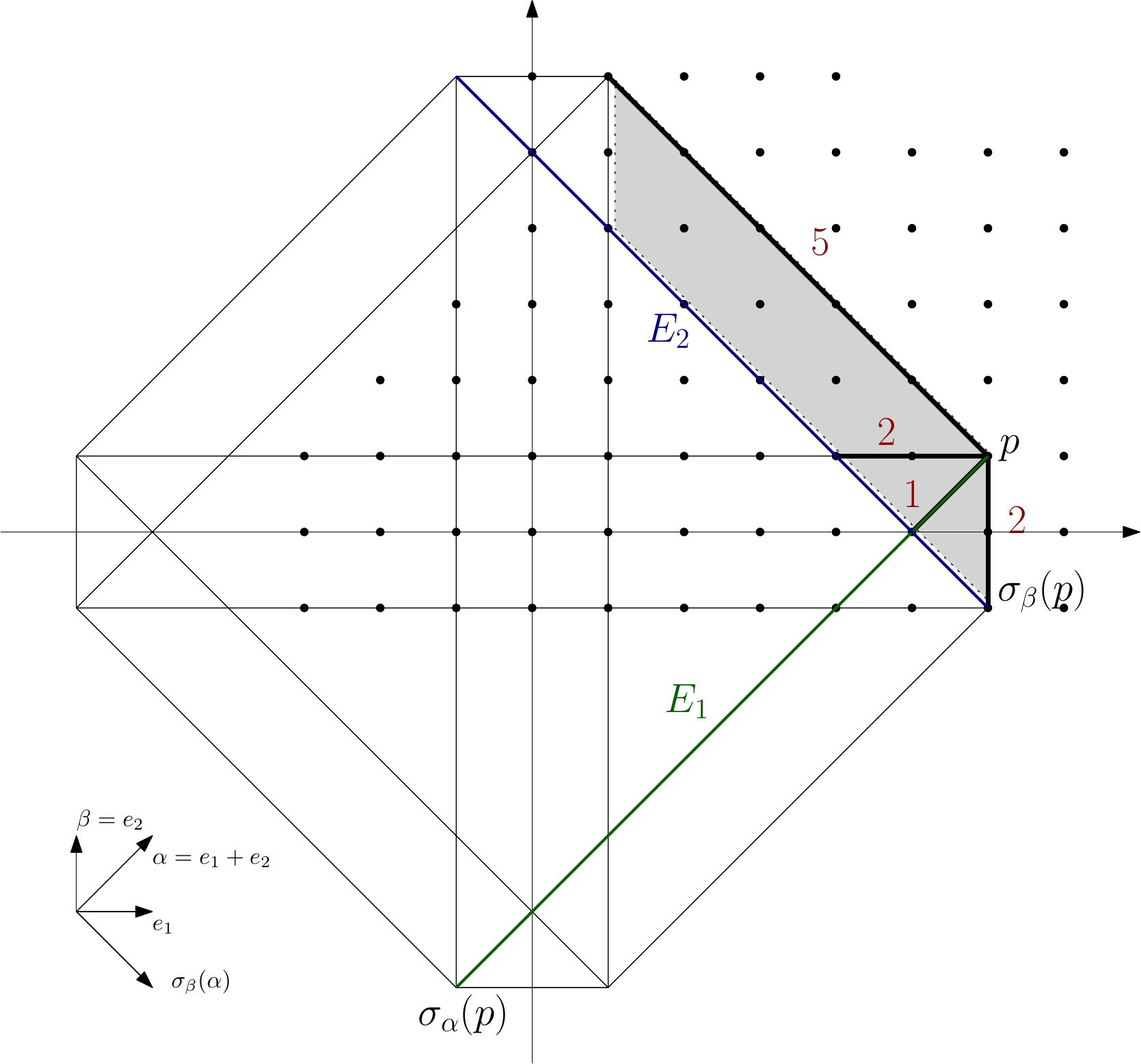}
\caption{One-skeleton of $SO(5)$ coadjoint orbit}\label{notlacedeg}
\end{figure}
Edge lengths are given with respect to the weight lattice. Preimage of the shaded region is the maximal subset centered around $p$ for the standard action of maximal torus. The portion of edge $E_1$ contained in this region is of length 
$$\left|\frac{\langle e_2, (6,1)\rangle}{\langle e_2, e_1+e_2\rangle}\right|=1.$$
Therefore using this centered region, we can construct embeddings of a ball of capacity at most $1$. Regions centered at the other fixed points would give the same result. The Theorem \ref{main} provides a better lower bound, because the pairings of $p$ with coroots $e_1^{\vee},e_2^{\vee},(e_1+e_2)^{\vee},(e_1-e_2)^{\vee}$ give (respectively): $12,\,2,\,7,\,5$ and minimum of this set is $2$.
\section{Proofs of Lemmas \ref{openodd} and \ref{openeven}}\label{polytopeproof}
{\bf Proof Lemma \ref{openodd}.}
We are given real numbers \begin{equation} \label{ineqodd} b_1 \geq a_1 \geq b_2 \geq a_2 \geq \ldots \geq a_{k-1} \geq b_{k} \geq |a_{k}|\end{equation}
and we are to show that there exist a real vector $Y=[y_1, \ldots, y_{2k}]^T$ in $\bb{R}^{2k}$ 
such that
the skew symmetric matrices 
\small
\begin{displaymath}
A:=\left(\begin{array}{c|c}
\begin{array}{cccc}
  L(a_1) &  & &\\
 &L(a_2)   & &\\
& &\ddots & \\
&&& L(a_k)
    \end{array} & Y\\
\hline
-Y^T &0 
\end{array}
\right)\textrm{  and  }S:=\left(\begin{array}{c|c}
\begin{array}{cccc}
  L(b_1) &  & &\\
 &L(b_2)   & &\\
& &\ddots & \\
&&& L(b_k)
    \end{array} & 0\\
\hline
0 &0 
\end{array}
\right).
\end{displaymath} \normalsize are in the same $SO(2k+1)$ orbit. 
\begin{proof} 
Two matrices in $\lie{so}(2k+1)^*$ are in the same $SO(2k+1)$ orbit if and only if they have the same characteristic polynomial.
The characteristic polynomial for $A$, $\chi_A(t)$ is
\small\begin{displaymath} \chi_A(t)=\left| \begin{array}{cccccc}
t &a_1&&&&-y_1\\-a_1 &t&&&&-y_2\\&& \ddots &&&\vdots\\ &&&t&a_k&-y_{2k-1}\\&&&-a_k&t&-y_{2k}\\y_1&y_2&\ldots&y_{2k-1}&y_{2k}&t\end{array} \right| \end{displaymath}
\begin{displaymath}
=-y_1\left(-a_1 \left| \begin{array}{cccccc} 0&t&a_2&&&\\0&-a_2&t &&&\\&&&\ddots&& \\ &&&&t&a_k\\&&&&-a_k&t\\ y_2&y_3&y_4&\ldots&y_{2k-1}&y_{2k}\end{array} \right|
-y_1\left| \begin{array}{cccccc} t&0&0&&&0\\0&t&a_2&&&\\0&-a_2&t &&&\\&&&\ddots&& \\ &&&&t&a_k\\0&&&&-a_k&t\end{array} \right|\right)+
\end{displaymath}\normalsize
\begin{displaymath}+y_2\left( -a_1\,(y_1 \prod_{j\neq 1}(t^2+a_j^2))\,+y_2\,(t\prod_{j\neq 1}(t^2+a_j^2))\,\right)+\ldots +t\prod_{j=1}^{k}(t^2+a_j^2)
\end{displaymath}
\begin{displaymath}
=(a_1y_1y_2+y_1^2t-a_1y_1y_2+y_2^2t)\,\prod_{j=2}^k(t^2+a_j^2)+\ldots + (y_{2k-1}^2+y_{2k}^2)\,t\,\prod_{j=1}^{k-1}(t^2+a_j^2)+t\prod_{j=1}^{k}(t^2+a_j^2)
\end{displaymath}
\begin{displaymath}
 =t\sum_{l=1}^{k}\,(y_{2l-1}^2+y_{2l}^2)\,\prod_{j\neq l}(t^2+a_j^2)+t\prod_{j=1}^{k}(t^2+a_j^2).\end{displaymath}
 The characteristic polynomial for $S$ is  $\chi_S(t)=t\,\prod_{j=1}^k(t^2+b_j^2).$
Simplifying $t$ we get the equation
\begin{equation}\label{charpolyodd}
\sum_{l=1}^{k}\,(y_{2l-1}^2+y_{2l}^2)\,\,\prod_{j\neq l}(t^2+a_j^2)+\prod_{j=1}^{k}(t^2+a_j^2)=\prod_{j=1}^k(t^2+b_j^2).
\end{equation}
{\it Case 1.} Assume first that $a$ and $b$ are regular, that is 
\begin{equation} \label{abregular} b_1 > a_1 > b_2 > a_2 > \ldots > a_{k-1} > b_{k} > |a_{k}|. \end{equation}
Then we can write the Equation \ref{charpolyodd} as 
$$\prod_{j=1}^{k}(t^2+a_j^2)\left( 1+\sum_{l=1}^{k}\,\frac{y_{2l-1}^2+y_{2l}^2}{t^2+a_l^2}\right)=\prod_{j=1}^k(t^2+b_j^2).$$
Substituting $t=\pm i b_s$ for $s=1, \ldots, k$ we get the system of equations
$$\forall_{s=1,\ldots,k}\,\,\,\,\,\left( 1+\sum_{l=1}^{k}\,\frac{y_{2l-1}^2+y_{2l}^2}{-b_s^2+a_l^2}\right)=0.$$
Introduce the notation
$$w_l=y_{2l-1}^2+y_{2l}^2.$$ Solving the Equation \ref{charpolyodd} for regular case is equivalent to finding nonnegative solution
in $w's$ to the system of linear conditions
\begin{equation}
 \forall_{s=1,\ldots,k}\,\,\,\,\,\sum_{l=1}^{k}\,\frac{w_l}{b_s^2-a_l^2}=1.
\end{equation}
Denote by $M=[m_{sl}]$, $m_{sl}=\frac{1}{b_s^2-a_l^2}$ the matrix of this system of equations.
Matrices of this type are called Cauchy matrices. In 1959 Schechter (\cite{S}) proved that 
$$\det M= \frac{\prod_{i=2}^k\prod_{j=1}^{i-1}(b_i^2-b_j^2)(a_i^2-a_j^2)}{\prod_{i=1}^k\prod_{j=1}^{k}(b_i^2-a_j^2)}\neq 0.$$
Moreover, he showed that the inverse matrix $M\inv =[m^{ij}]$ is given by the formula
$$m^{ij}=(b_j^2-a_i^2)B_j(a_i^2)A_i(b_j^2)$$
where $B_j(x),A_i(x)$ are the Lagrange polynomials for $(b_i^2)$ and $(a_j^2)$. This means that 
$$A_i(x)=\frac{A(x)}{A'(a_i^2)(x-a_i^2)} \textrm{  and  } B_i(x)=\frac{B(x)}{B'(b_i^2)(x-b_i^2)},$$
with $$A(x)=\prod_{i=1}^k(x-a_i^2) \textrm{  and  } B(x)=\prod_{i=1}^k(x-b_i^2).$$
Therefore, the solution to our system is given by (see also \cite[Ch VIII]{M})
$$w_l=-\,\frac{\prod_{j=1}^n(a_l^2-b_j^2)}{\prod_{j\neq l,\,j=1}^n(a_l^2-a_j^2)}.$$
Notice that, due to inequalities \ref{abregular}, the numerator is positive if and only if $\sharp \{j;\,j\geq l\}$ is even, while the denominator is positive if and only if
$\sharp \{j;\,j > l\}$ is even. Thus $w_l$ is always positive, as required.

If the inequalities \ref{ineqodd} are not satisfied, then some $w_l$ is negative and therefore there is no solution in $y$'s.

{\it Case 2.} Suppose that $b$ is regular but $a$ is not, that is there exists $j_0$ such that $a_{j_0}=b_m$ (that is $m=j_0$ or $j_0+1$).

Suppose for a moment that $a_{j_0}$ is the only coordinate of $a$ that is equal to $b_m$, that is, $b_m \neq a_j$ for all $j \neq j_0$. Then, substituting $t=ib_m$ in Equation (\ref{charpolyodd}), we get that 
$$w_{j_0}\, \prod_{j \neq j_0}(a_j^2-b_m^2)=0,$$
thus $w_{j_0}=0$. Therefore $y_{2j_0-1}=y_{2j_0}=0$. 
This means that ever term in Equation \ref{charpolyodd} contains a factor $(t^2+b_m^2)$ and we can simplify this factor.
Then we arrive at the equation with just $k-1$ variables $w_1, \ldots, \widehat{w_{j_0}}, \ldots w_k$ and $2k-2$ parameters which are now regular
or at least less degenerate. Repeating this step if necessary, we get to the equation similar to Equation (\ref{charpolyodd}) that is regular
(and has less variables and parameters).

Now suppose that $a_{j_0}$ is not the only coordinate of $a$ that is equal to $b_m$. As $b$ is regular, this can happen if and only if 
$a_{m-1}=b_m=a_m$. Now every term in Equation \ref{charpolyodd} contains a factor $(t^2+b_m^2)$. 
We simplify this factor. Introducing new variables and parameters for $j=1, \ldots, k-1$\footnotesize 
\begin{displaymath}
 \widetilde{a_j}=\begin{cases}a_j & j < m\\
              a_{j+1} & j \geq m
             \end{cases},\,\,\,
\widetilde{b_j}=\begin{cases}b_j & j < m\\
              b_{j+1} & j \geq m
             \end{cases},\,\,\,
\widetilde{w_j}=\begin{cases}w_j & j < m-1\\
w_{m-1}+w_m&j=m-1\\
              w_{j+1} & j >m-1
             \end{cases}
\end{displaymath}
\normalsize
we get the equation
$$\sum_{l=1}^{k-1}\,(\widetilde{w_l})\,\,\prod_{j\neq l}(t^2+\widetilde{a_j}^2)+\prod_{j=1}^{k}(t^2+\widetilde{a_j}^2)=\prod_{j=1}^k(t^2+\widetilde{b_j}^2),$$
which is regular or at least less degenerate then the one we started with. Repeating the above steps if necessary, we obtain a regular equation and can find the solution using the inverse of appropriate Cauchy matrix.

{\it Case 3.} Now we deal with the case of $b$ non-regular. Again we will try to reduce it, step by step, to the regular case. Suppose that $b_j=b_{j+1}$ for some index $j$. Then $a_j$ is forced by the inequalities (\ref{polytopeineqodd}) to be also equal to $b_j$.

If no other $a_l$ is equal to $a_j$, then substituting $t=ib_j$ we obtain that $w_j=0$. Therefore $y_{2j-1}=y_{2j}=0$. This means that every term in the Equation (\ref{charpolyodd}) contains the factor $(t^2 +b_j^2)$. Simplifying this factor we arrive at the equation that is one step less degenerate.

If there are other $a_l$ also equal to $a_j$, then every term in the Equation (\ref{charpolyodd}) contains the factor $(t^2 +b_j^2)$. We can simplify this factor and, similarly to the case above, introduce new variables to obtain an equation that is one step less degenerate.\\
\indent It is clear from the proof that if there exists unique index $j$ such that $a_j=b_m$, then $y_{2j-1}=y_{2j}=0$.
\end{proof}
{\bf Proof of Lemma \ref{openeven}.} Now we proof the even dimensional analogue, that is Lemma \ref{openeven}. We are given real numbers \begin{equation}\label{ineqeven}
 a_1 \geq b_1 \geq a_2 \geq b_2\geq \ldots \geq b_{k-1} \geq |a_{k}|\end{equation} and we are to find a real vector $Y=[y_1, \ldots, y_{2k-1}]^T$ in $\bb{R}^{2k-1}$ 
such that
the skew symmetric matrices 
\begin{displaymath}
A:=\left(\begin{array}{c|c}
\begin{array}{ccccc}
  L(b_1) &  && &\\
 &L(b_2)   && &\\
& &\ddots & &\\
&&& L(b_{k-1})&\\&&&&0
    \end{array} & Y\\
\hline
-Y^T &0 
\end{array}
\right)\textrm{  and  }S:=\left(
\begin{array}{cccc}
  L(a_1) &  & &\\
 &L(a_2)   & &\\
& &\ddots & \\
&&& L(a_k)
    \end{array} 
\right).
\end{displaymath} are in the same $SO(2k)$ orbit. 

If two matrices in $\lie{so}(2k)^*$ are in the same $SO(2k)$ orbit, then in particular they have the same characteristic polynomial. We could proceed as in the odd dimensional case and start with comparing the characteristic polynomials of $A$ and $S$. This would again involve, for regular case, solving some linear system of equations, with unknowns $\{y_{2l-1}^2+y_{2l}^2, y_{2k-1}\}$, given by a Cauchy matrix. 
By the result of Schechter we know the inverse matrix, but it is still computationally challenging to show that the solution is nonnegative (except possibly at $y_{2k-1}$). For this reason, and to present another approach, we will proceed differently. We will transform the problem into a problem for the unitary case and use the same theorems that were used in \cite{P}. In particular we use the following Lemma, which is a slight strengthening of Lemma 3.6 in \cite{P}(Lemma 3.5 in \cite{NNU}, see also \cite{GS2}).
\begin{lemma}\label{open}
 For any real numbers $\mu_1 \geq \nu_1 \geq \mu_2 \geq \ldots \geq \mu_{2k-1} \geq \nu_{2k-1} \geq \mu_{2k}$ there exist $x_1, \ldots, x_{2k-1}$ in $\bb{C}$ and $x_{2k}$ in $\bb{R}$ such that
the Hermitian matrix  
\begin{displaymath}
A:=\left(\begin{array}{cccc}
    \nu_1&  & 0& \bar{x}_1\\
& \ddots & & \vdots\\
0& & \nu_{2k-1} & \bar{x}_{2k-1}\\
x_1& \hdots & x_{2k-1}& x_{2k}
    \end{array}
\right),
\end{displaymath}
has eigenvalues $\mu_1, \ldots , \mu_{2k}$. The solution is not unique: only the values $|x_1|, \ldots, |x_{2k-1}|$ and $x_{2k}$ are uniquely defined. Inequalities between $\mu_j$ and $\nu_j$ are necessary for such $x_1, \ldots, x_{k+1}$ to exist. Moreover \\
1. If $m$ is the unique index such that $\mu_j=\nu_m$ then $x_m=0$.\\
2. Suppose that $\nu_l=-\nu_{2k-l},\,\mu_l=-\mu_{2k+1-l},$ for $l=1,\dots, k$,
(so $\nu_{k}=0$). Then $|x_l|=|x_{2k-l}|$ for $l=1,\dots, k$ and $x_{2k}=0$.
\end{lemma}
\begin{proof} Here we only prove the additional, strengthening statements $1$ and $2$.

 1. The characteristic polynomial of matrix $A$ is
$$t\prod_{l=1}^{2k-1}(t-\nu_l)-\sum_{i=1}^{2k-1} |x_i|^2 \prod_{l\neq i}(t-\nu_l).$$
This must be equal to $\prod_{l=1}^{2k}(t-\mu_l),$ the characteristic polynomial of $S$. 
Therefore, substituting $t=\mu_j$ we get
$$0=t\prod_{l=1}^{2k-1}(\mu_j-\nu_l)-\sum_{i=1}^{2k-1} |x_i|^2 \prod_{l\neq i}(\mu_j-\nu_l)=-|x_m|^2\prod_{l\neq m}(\mu_j-\nu_l).$$
This means that $x_m=0$, because $m$ is the unique index such that $\mu_j=\nu_m$.

2. The trace of $A$ is $0=\sum_{l=1}^{2k} \mu_l=\sum_{l=1}^{2k-1} \nu_l+x_{2k}$, thus $x_{2k}=0.$ 
Notice that conjugating $A$ with a matrix of permutation switching $l$ with $2k-l$, for $l=1, \ldots, k$,
(which is in $U(n)$), will give the matrix $A'$, with the same eigenvalues as $A$.
\begin{displaymath}
A':=\left(\begin{array}{cccc}
    \nu_{2k-1}&  & 0& \bar{x}_{2k-1}\\
& \ddots & & \vdots\\
0& & \nu_1 & \bar{x}_1\\
x_{2k-1}& \hdots & x_1& 0
    \end{array}
\right)=
\left(\begin{array}{cccc}
   - \nu_1&  & 0& \bar{x}_{2k-1}\\
& \ddots & & \vdots\\
0& & -\nu_{2k-1} & \bar{x}_1\\
x_{2k-1}& \hdots & x_1& 0
    \end{array}
\right)
\end{displaymath}
Eigenvalues of $(-A')$ are $\{-\mu_l;\, l=1,\ldots 2k\}=\{\mu_l;\,l=1,\ldots 2k\}$, the same as of the matrix $A$.
Therefore the sequence $(-x_{2k-1},\ldots, -x_1,0)$ is also a solution to question in the Lemma \ref{open}.
For such a solution the absolute values are uniquely defined. Therefore $|x_l|=|-x_{2k-l}|=|x_{2k-l}|$ for $l=1, \ldots, k$.
\end{proof}
Now we are ready to prove Lemma \ref{openeven}.\begin{proof}
 Applying the Lemma \ref{open} we get that there exists $X=(x_1, \ldots, x_{2k-1}) \in \bb{C}^{2k-1}$, such that the matrix\footnotesize
\begin{displaymath}
\left(\begin{array}{ccccccccc|c}
b_1&&&&&&&&&\overline{x}_1\\ 
&b_2&&&&&&&&\overline{x}_3\\ 
&&\ddots&&&&&&&\vdots\\
&&&b_{k-1}&&&&&&\overline{x}_{2k-3}\\
&&&&0&&&&&\overline{x}_{2k-1}\\
&&&&&-b_{k-1}&&&&\overline{x}_{2k-2}\\
&&&&&&\ddots&&&\vdots\\
     &&&&&&&-b_2&&\overline{x}_4\\ 
&&&&&&&&-b_1&\overline{x}_2\\           
\hline
x_1&x_3&\hdots&x_{2k-3}&x_{2k-1}&x_{2k-2}&\hdots&x_4&x_2&0
\end{array}
\right)
\end{displaymath}\normalsize has eigenvalues $(a_1,\ldots,|a_{k}|,-|a_{k}|,\ldots,-a_1)$, and $|x_{2j-1}|=|x_{2j}|$ for $j=1,\ldots, k-1$.
Conjugating with a permutation matrix (which is also in $U(2k)$) will not change the eigenvalues. Therefore there exist a matrix $B \in U(2k)$ such that\scriptsize
\begin{equation}\label{b6}
B\,\left(\begin{array}{cccccc|c}
b_1&&&&&&\overline{x}_1\\ 
&-b_1&&&&&\overline{x}_2\\ 
&&\ddots&&&&\vdots\\
&&&b_{k-1}&&&\overline{x}_{2k-3}\\
&&&&-b_{k-1}&&\overline{x}_{2k-2}\\
&&&&&0&\overline{x}_{2k-1}\\
\hline
x_1&x_2&\hdots&x_{2k-3}&x_{2k-2}&x_{2k-1}&0
\end{array}
\right)\,B\inv=\left(\begin{array}{ccccc}
a_1&&&&\\ 
&-a_1&&&\\ 
&&\ddots&&\\
&&&a_k&\\
&&&&-a_k
\end{array}
\right)
\end{equation}\normalsize Notice that 
\begin{displaymath}
\left(
 \begin{array}{cc}
1&i \\
i&1
\end{array}
\right)\,
\left(
 \begin{array}{cc}
 0 &-x\\
x&0
\end{array}
\right)\,
\left(
 \begin{array}{cc}
 1&-i \\
-i&1
\end{array}
\right)\,=2\,
\left(
 \begin{array}{cc}
 ix&0\\
0&-ix
\end{array}
\right).
\end{displaymath}
Define the matrices $J_m \in U(2m)$, $L_m\in U(2m+1)$ in the following way \small
\begin{displaymath}
J_m:=\frac{1}{\sqrt{2}}\left(
 \begin{array}{ccc}
\begin{array}{cc}
1&i \\
i&1
\end{array}&&\\
&\ddots&\\
&&\begin{array}{cc}
1&i \\
i&1
\end{array}
  \end{array}
 \right),\,\,\,
L_m:=\left(
 \begin{array}{c|c}
J_m&0 \\
\hline
0&1
\end{array}
 \right).
\end{displaymath} \normalsize
We will surpress $m$ from the notation when the dimension is understood. 
Have \small
\begin{displaymath}
J\,\left( \begin{array}{ccc}
        L(a_1)&&\\
&\ddots&\\
&&L(a_k) \end{array}\right)\,J\inv=
\left(\begin{array}{ccccc}
ia_1&&&&\\
&-ia_1&&&\\
&&\ddots&&\\
&&&ia_k&\\
&&&&-ia_k
\end{array}
\right).
\end{displaymath} \normalsize Also \small \begin{displaymath}
i\,\left(\begin{array}{cccccc|c}
b_1&&&&&&\overline{x}_1\\ 
&-b_1&&&&&\overline{x}_2\\ 
&&\ddots&&&&\vdots\\
&&&b_{k-1}&&&\overline{x}_{2k-3}\\
&&&&-b_{k-1}&&\overline{x}_{2k-2}\\
&&&&&0&\overline{x}_{2k-1}\\
\hline
x_1&x_2&\hdots&x_{2k-3}&x_{2k-2}&x_{2k-1}&0
\end{array}
\right)\end{displaymath}
\begin{displaymath}
=\left(\begin{array}{c|c} L&\\ \hline&1\end{array}\right)
\left(\begin{array}{c|c} L\inv &\\ \hline&1\end{array}\right)\,i\,
\left(\begin{array}{c|c} \begin{array}{cccc}b_1&&&\\&\ddots&&\\&&-b_{k-1}&\\&&&0 \end{array}&X^*\\ \hline X&0\end{array}\right)\, \,\left(\begin{array}{c|c} L &\\ \hline&1\end{array}\right)\,\left(\begin{array}{c|c} L \inv&\\ \hline&1\end{array}\right)=
\end{displaymath}
\begin{displaymath}
=\left(\begin{array}{c|c} L&\\ \hline&1\end{array}\right)
\left( \begin{array}{c|c}\begin{array}{cccc}L(b_1)&&&\\ &\ddots&& \\ && L(b_{k-1})&\\&&&0 \end{array}&i L\inv\,X^*\\
 \hline iX\,L &0 \end{array}\right) \,\left(\begin{array}{c|c} L\inv&\\ \hline&1\end{array}\right)=
\end{displaymath}
\begin{displaymath}
=\left(\begin{array}{c|c} L&\\ \hline&1\end{array}\right)\,\,A\,\,\left(\begin{array}{c|c} L\inv&\\ \hline&1\end{array}\right)\end{displaymath}
\normalsize where \small
\begin{displaymath}
 A:=\left( \begin{array}{c|c}\begin{array}{cccc}L(b_1)&&&\\ &\ddots&& \\ && L(b_{k-1})&\\&&&0 \end{array}&i L\inv\,X^*\\
 \hline iX\,L &0 \end{array}\right) 
\end{displaymath}
\normalsize Together with Equation \ref{b6} this gives that \small \begin{displaymath}
 S=\left( \begin{array}{ccc}
        L(a_1)&&\\
&\ddots&\\
&&L(a_k) \end{array}\right) =J\inv \, i \,\left(\begin{array}{ccccc}
a_1&&&&\\ 
&-a_1&&&\\ 
&&\ddots&&\\
&&&a_k&\\
&&&&-a_k
\end{array}
\right)
\, J   \end{displaymath} \normalsize \begin{displaymath} 
=J\inv \,B\,\left(\begin{array}{c|c} L&\\ \hline&1\end{array}\right)\,\,A\,\,\left(\begin{array}{c|c} L\inv&\\ \hline&1\end{array}\right) \,B\inv \,J                                                                  
\end{displaymath}
Notice that we can choose $X$ so that $A$ is not only in $\lie{u}(2k)^*$ but also in $\lie{so}(2k)^*$. If $x_j=r_j+iw_j$, then
\begin{displaymath}
 Y:=i\,L\inv \,X^*=\frac{1}{\sqrt{2}}\left(
 \begin{array}{cccc}
\begin{array}{cc}
1&-i \\
-i&1
\end{array}&&&\\
&\ddots&&\\
&&\begin{array}{cc}
1&-i \\
-i&1
\end{array}&\\
&&&1
  \end{array}
 \right)\,\left(\begin{array}{c}
w_1+ir_1\\
w_2+ir_2\\
\ldots \\
w_{2k-3}+ir_{2k-3}\\
w_{2k-2}+ir_{2k-2}\\
w_{2k-1}+ir_{2k-1}
           \end{array}\right)=\end{displaymath}\begin{displaymath}=\frac{1}{\sqrt{2}}\left(\begin{array}{c}
w_1+r_2+i(r_1-w_2)\\
w_2+r_1+i(r_2-w_1)\\
\ldots \\
w_{2k-3}+r_{2k-2}+i(r_{2k-3}-w_{2k-2})\\
w_{2k-2}+r_{2k-3}+i(r_{2k-2}-w_{2k-3})\\
w_{2k-1}+ir_{2k-1}
           \end{array}\right).
\end{displaymath}
This vector is real if and only if $r_{2j-1}=w_{2j}$ and $r_{2j}=w_{2j-1}$, for $j=1, \ldots, k-1$ and $r_{2k-1}=0$. According to Lemma \ref{open}, only the absolute values of $x_j$'s are uniquely defined
and $|x_{2j-1}|=|x_{2j}|$ for $j=1, \ldots, k-1$. Therefore, if we take any $x_{2j-1}=r_{2j-1}+iw_{2j-1}$ with prescribed absolute value, and put $x_{2j}=w_{2j-1}+ir_{2j-1}$, $x_{2k-1}=|x_{2k-1}|$ 
then vectors $i\,L_k\inv \,X^*$ and its transpose conjugate $-i \,X \,L_k$ are real and $A \in \lie{so}(2k)^*$.\\
Moreover, the only two matrices in the positive Weyl chamber with the same characteristic polynomial as the matrix $A$ are \begin{displaymath}
S=\left(\begin{array}{ccc}
L(a_1)&&\\ 
&\ddots&\\
 &&L(a_k)    \end{array}\right),\,\widetilde{S}:=\left(\begin{array}{ccc}
L(a_1)&&\\ 
&\ddots&\\
 &&L(-a_k)    \end{array}\right) .\end{displaymath}
These matrices are $O(2k)$ conjugate but not $SO(2k)$ conjugate. Let $R \in O(2k)$ denote the diagonal matrix
with all $1$'s on diagonal except the last, $2k$-th, entry that is equal to $-1$.
Then $$\widetilde{S}=R\,S\,R\inv .$$
If the matrix $A$ we have constructed is in fact in the $SO(2k)$ orbit through $\widetilde{S}$, then the matrix 
$$R\,A\,R\inv=\left(\begin{array}{c|c}
\begin{array}{ccccc}
  L(b_1) &  && &\\
 &L(b_2)   && &\\
& &\ddots & &\\
&&& L(b_k)&\\&&&&0
    \end{array} &- Y\\
\hline
Y^T &0 
\end{array}
\right)$$
is in the $SO(2k)$ orbit through $S$. Therefore, if $Y$ is the vector such that matrices $A$ and $S$ have the same characteristic polynomial,
then either $Y$ or $-Y$ is the solution we need. Again we have that $y_{2j-1}^2+y_{2j}^2=2r_{2j-1}^2+2w_{2j}^2=2|x_{2j-1}|^2$ and $y_{2k-1}=\pm |x_{2k-1}|$ are uniquely defined.
\end{proof}

\end{document}